\documentclass[11pt]{amsart}
\usepackage{graphicx,color,amsmath,amssymb,mathrsfs}
\usepackage{amsfonts,amsthm,epsfig,epstopdf}
\usepackage[margin=1.0in]{geometry}
\usepackage[english]{babel}

\theoremstyle{plain}
\newtheorem{theorem}{Theorem}[section]

\newtheorem{lemma}[theorem]{Lemma}
\newtheorem{proposition}[theorem]{Proposition}

\theoremstyle{definition}

\newtheorem{remark}[theorem]{Remark}
\numberwithin{theorem}{section}
\numberwithin{equation}{section}
\numberwithin{figure}{section}

\def\N{\mathbb N}

\def\R{\mathbb R}
\def\L{\mathscr L}
\def\Prob{\mathcal P}

\def\1{{\bf 1}}
\def\j{\mathrm{j}}

\def\a{\alpha}
\def\b{\beta}

\def\de{\delta}
\def\e{\varepsilon}

\def\l{\lambda}
\def\L{\Lambda}

\def\vphi{\varphi}
\def\nab{\nabla}
\def\t{\theta}

\def\Om{\Omega}

\newcommand{\ep}{\varepsilon}

\renewcommand{\H}{\mathcal{H}}

\newcommand{\V}{\tilde{V}}
\newcommand{\eone}{\text{e}_{1}}
\newcommand{\ent}{\mathrm{Ent}}

\DeclareMathOperator{\Id}{Id}

\DeclareMathOperator{\trace}{tr}

\DeclareMathOperator{\cost}{cost}

\begin{document}
\title[Lipschitz Changes of Variables]{Lipschitz Changes of Variables between\\ Perturbations of Log-concave Measures}

\author[M.\ Colombo]{Maria Colombo}
\address{
Institute for Theoretical Studies, ETH Z\"urich, Clausiusstrasse 47, CH-8092 Z\"urich, Switzerland
\\
Institut f\"ur Mathematik, Universitaet Z\"urich, Winterthurerstrasse 190, CH-8057 Z\"urich,
Switzerland}
\email{maria.colombo@math.uzh.ch}

\author[A.\ Figalli]{Alessio Figalli}
\address{The University of Texas at Austin, 
Mathematics Dept. RLM 8.100,
2515 Speedway Stop C1200,
Austin, TX 78712, USA}
\email{figalli@math.utexas.edu}

\author[Y.\ Jhaveri]{Yash Jhaveri}
\address{The University of Texas at Austin, 
Mathematics Dept. RLM 8.100,
2515 Speedway Stop C1200,
Austin, TX 78712, USA}
\email{yjhaveri@math.utexas.edu}

%

\begin{abstract}
Extending a result of Caffarelli, we provide global Lipschitz changes of variables between compactly supported perturbations of log-concave measures.
The result is based on a combination of ideas from optimal transportation theory and a new Pogorelov-type estimate. In the case of radially symmetric measures, Lipschitz changes of variables are obtained for a much broader class of perturbations.
\end{abstract}

\maketitle

\section{Introduction}
In \cite{c2}, Caffarelli built Lipschitz changes of variables between log-concave probability measures. More precisely, he showed that if $V,W \in C^{1,1}_{\rm loc}(\R^n)$ are convex functions with $D^{2}V(x) \leq \L_{V}\Id$ and $\l_{W}\Id \leq D^{2}W(x)$ for a.e. $x \in \R^{n}$ with $0 < \L_V,\lambda_W < \infty$, then there exists a Lipschitz map $T: \R^n \to \R^n$ such that $T_\# \big( e^{-V(x)} \, dx \big) =  e^{-W(x)} \, dx$ \footnote{Given two finite Borel measures $\mu$ and $\nu$ and a Borel map $T:\R^n \to \R^n$, recall that $T_{\#}\mu = \nu$ if
\[ \int_{\R^{n}} \vphi(y) \, d\nu(y) =  \int_{\R^{n}} \vphi(T(x)) \, d\mu(x) \qquad \forall \,\vphi \text{ Borel and bounded}. \]
} and
\begin{equation}
\label{eq:bound DT}
\| \nab T \|_{L^\infty(\R^{n})} \leq \sqrt{\L_{V}/\l_{W}}.
\end{equation}
The map $T$ is obtained via optimal transportation.
It is the unique solution of the Monge problem for quadratic cost:
\begin{equation*}
\min \bigg\{ \int_{\R^{n}} |x-T(x)|^{2} e^{-V(x)} \, dx: T_\# \big( e^{-V(x)} \, dx \big) =  e^{-W(x)} \, dx\bigg\} 
\end{equation*}
(see Section~\ref{sec:Preliminaries} for more details,
and \cite{KimMil} for a completely different construction of a Lipschitz change of variables in this setting). We note that a particularly important feature of Caffarelli's result is that the bound \eqref{eq:bound DT}
is {\em independent} of the dimension $n$.

A consequence of Caffarelli's result is the possible deduction of certain functional inequalities (such as log-Sobolev or Poincar\'e-type inequalities) for log-concave measures from their corresponding Gaussian versions. 
For instance, denoting the standard Gaussian measure on $\R^n$ by $\gamma_n$, consider the Gaussian log-Sobolev inequality,
$$\int_{\R^n} f^2 \ln f \, d \gamma_n \leq 
\int_{\R^n} |\nabla f|^2 \, d \gamma_n + \bigg( \int_{\R^n} f^2 \, d \gamma_n\bigg) \ln \bigg( \int_{\R^n} f^2 \, d \gamma_n\bigg),$$
which holds for every function $f\in W^{1,2}(\R^n)$. For any measure $\nu$ such that there exists a Lipschitz change of variables between $\nu$ and the Gaussian measure, namely $\nu = T_\# \gamma_n$, we deduce, applying the change of variable formula twice, that
\begin{equation*}
\begin{split}
\int_{\R^n} f^2 \ln f \, d \nu &= 
\int_{\R^n} f(T)^2 \ln f(T) \, d \gamma_n
\\
& \leq 
\int_{\R^n} |\nabla[f\circ T]|^2 \, d \gamma_n + \bigg( \int_{\R^n} f(T)^2 \, d \gamma_n\bigg) \ln \bigg( \int_{\R^n} f(T)^2 \, d \gamma_n\bigg)
\\
& \leq \| \nab T \|_{L^\infty(\R^{n})}^2
\int_{\R^n} |\nabla f (T)|^2 \, d \gamma_n + \bigg( \int_{\R^n} f(T)^2 \, d \gamma_n\bigg) \ln \bigg( \int_{\R^n} f(T)^2 \, d \gamma_n\bigg)
\\
& = \| \nab T \|_{L^\infty(\R^{n})}^2
\int_{\R^n} |\nabla f |^2 \, d \nu + \bigg( \int_{\R^n} f^2 \, d \nu\bigg) \ln \bigg( \int_{\R^n} f^2 \, d \nu\bigg).
\end{split}
\end{equation*}
Therefore, $\nu$ enjoys a log-Sobolev inequality with constant $\| \nab T \|_{L^\infty(\R^{n})}^2$.

Besides the natural consequences described in \cite{c2} and above, Caffarelli's Theorem has found numerous applications in various fields: indeed, it can be used to transfer isoperimetric inequalities, to obtain correlation inequalities, and more (see, for instance, 
\cite{c-e,cor,har,kla}).
\\

In this paper, we extend the result of Caffarelli by building Lipschitz changes of variables between perturbations of $V$ and $W$ that are not necessarily convex.
Perturbations of log-concave measures (in particular, perturbations of Gaussian measures) appear, for instance, in quantum physics as a means to help understanding solutions to physical theories with nonlinear equations of motion.
In cases where an explicit solution is unknown, perturbations of log-concave measures can be used to yield approximate solutions.

We let $\Prob(X)$ denote the space of probability measures on a metric space $X$.
The main result of the paper is the following: 

\begin{theorem}
\label{thm:main}
Let $V \in C^{1,1}_{\rm loc}(\R^{n})$ be such that $e^{-V(x)} \, dx \in \Prob(\R^{n})$.
Suppose that $V(0) = \inf_{\R^{n}} V$ and there exist constants $0 < \l,\, \L < \infty$ for which $\l\Id \leq D^{2}V(x) \leq \L\Id$ for a.e. $x \in \R^{n}$.
Moreover, let $R>0$, $q \in C^{0}_{c}(B_{R})$, and $c_q \in \R$ be such that $e^{-V(x)+c_q-q(x)} \, dx \in \Prob(\R^{n})$.
Assume that $-\l_{q}\Id \leq D^{2} q$ in the sense of distributions for some constant $\l_{q} \geq 0$. 
Then, there exists a constant $C = C(R,\l,\L,\l_{q})>0$, independent of $n$, such that the optimal transport  map $T$ that takes $e^{-V(x)} \, dx$ to $e^{-V(x)+c_q-q(x)} \, dx$ satisfies
\begin{equation}
\label{ts:lip-transp}
\| \nab T\|_{L^{\infty}(\R^{n})} \leq C.
\end{equation}
\end{theorem}
The crucial point here is that the estimate on the Lipschitz constant of the optimal transport map is independent of dimension, as it is in Caffarelli's results for log-concave measures.

In the case of spherically symmetric measures, we are able to weaken the assumptions on both the log-concave measure and its perturbation and still obtain a global Lipschitz change of variables. In particular, the Lipschitz constant is controlled only by the $L^\infty$-norm of the positive and negative parts of the perturbation $q$, denoted by $q^+$ and $q^-$. In the following theorem, we first analyze the $1$-dimensional problem:

\begin{theorem}
\label{1-D noncompact}
Let $V : \R \to \R \cup \{ \infty \}$ be a convex function and $q:\R\to \R$ be a bounded function such that $e^{-V(x)} \,dx,
\, e^{-V(x)-q(x)} \,dx \in \Prob(\R)$. Then, the optimal transport $T$ that takes $e^{-V(x)} \,dx$ to $e^{-V(x)-q(x)} \,dx$ is Lipschitz and satisfies
\begin{equation}
\label{ts: T-lip-q}
\|\log T'\|_{L^{\infty}(\R)} \leq \|q^{+}\|_{L^{\infty}(\R)}+\|q^{-}\|_{L^{\infty}(\R)}.
\end{equation}
\end{theorem}
We remark that while the map $T$ in Theorem~\ref{1-D noncompact} is only unique up to sets of $e^{-V(x)}\,dx$-measure zero, arguing by approximation, we can find a particular transport $T$ for which the estimate on $\log T'$ in \eqref{ts: T-lip-q} is satisfied almost everywhere in $\R$.
Applying this $1$-dimensional result to radially symmetric densities, we obtain the following:

\begin{theorem}\label{thm:sym}
Let $V : \R^{n} \to \R \cup \{ \infty \}$ be a convex, radially symmetric function and $q:\R^n \to\R$ be a bounded, radially symmetric function such that $e^{-V(x)} \,dx,\, e^{-V(x)-q(x)} \,dx \in \Prob(\R^{n})$. Then, the optimal transport $T$ that takes $e^{-V(x)} \,dx$ to $e^{-V(x)-q(x)} \,dx$ is Lipschitz and satisfies
\begin{equation}
\label{ts: S-lip}
e^{-\|q^{+}\|_{L^{\infty}(\R^{n})}-\|q^{-}\|_{L^{\infty}(\R^{n})}}\Id \leq \nab T(x) \leq e^{\|q^{+}\|_{L^{\infty}(\R^{n})}+\|q^{-}\|_{L^{\infty}(\R^{n})}}\Id \qquad \text{for a.e. } x \in \R^{n}.
\end{equation}
\end{theorem}
Note that the assumption $e^{-V(x)-q(x)} \,dx \in \Prob(\R^{n})$ in Theorems~\ref{1-D noncompact} and \ref{thm:sym}, unlike in Theorem~\ref{thm:main}, is nonrestrictive.
Since $q$ is not required to be compactly supported, the normalization constant making $e^{-V(x)-q(x)} \,dx$ a probability measure if it were not already can simply be absorbed into $q$.

We further remark that the $1$-dimensional estimate in Theorem \ref{1-D noncompact} is false in higher dimensions when one does not assume that the densities are radially symmetric. More precisely, taking the reference measure $e^{-V(x)} \, dx$ to be the standard Gaussian measure, the estimate 
\begin{equation}
\label{eqn:ts-linearized-n}
\|D^2\phi-\Id\|_{L^{\infty}(\R^n)} \leq C\|q\|_{L^{\infty}(\R^n)}
\end{equation}
cannot be true for $n > 1$ (see Remark~\ref{rmk:lin-est} to understand the relationship between \eqref{ts: T-lip-q} and \eqref{eqn:ts-linearized-n} for $n=1$). This is manifest if we recall that the Monge-Amp\`{e}re equation linearizes to the Poisson equation, which does not enjoy $C^{1,1}_{\rm{loc}}$ estimates for bounded right-hand side. In other words, given $V$ and $q$ to be chosen, letting $\phi_{\ep}$ be the potential such that $\nab \phi_{\ep}$ takes $e^{-V(x)} \, dx$ to $e^{-V(x)-\ep q(x)} \, dx$ (for simplicity, we omit the normalization constant that makes $e^{-V(x)-\ep q(x)} \, dx$ a probability measure) and setting $\psi_{\ep}(x) = (\phi_{\ep}(x)-|x|^{2}/2)/\ep$, we have that
\[  \Delta \psi_{\ep}+O(\ep) 
= \frac{\log\det \nab^{2}\phi_{\ep}}{\ep} 
=\frac{-V+V(\nab \phi_{\ep})+\ep q(\nabla \phi_{\ep})}{\ep} 
= \langle x, \nabla \psi_{\ep} \rangle +q(\nab \phi_{\ep}) +O(\ep) \]
for every $\ep>0$.
The estimate \eqref{eqn:ts-linearized-n} implies that $\sup_{\ep > 0} \|D^{2}\psi_{\ep}\|_{L^{\infty}(\R^{n})} < \infty$ and, therefore, the existence of a $C^{1,1}_{\rm{loc}}$ solution to the Poisson equation with bounded right-hand side, an impossibility in higher dimensions.

Although this heuristic argument is convincing, the details of the proof are rather delicate, and we give them in the Appendix for completeness. \\

\noindent{\bf Acknowledgments.} M. Colombo acknowledges the support
of the {Gruppo Nazionale per l'Analisi Matematica, la Probabilit\`a e le loro Applicazioni (GNAMPA)} of the {Istituto Nazionale di Alta Matematica (INdAM)}. 
A. Figalli has been partially supported 
by NSF Grant DMS-1262411 and NSF Grant DMS-1361122.
Y. Jhaveri would like to thank Pablo Stinga for helpful conversations.
Part of this work was done while the authors were guests of the FIM at ETH Z\"urich in the Fall of 2014; the hospitality of the Institute is gratefully acknowledged.

\section{Preliminaries}
\label{sec:Preliminaries}
We begin with some preliminaries on optimal transportation and the Monge-Amp\`{e}re equation, and we fix some notation. 

Let $\mu, \nu \in \Prob(\R^{n})$. The Monge optimal transport problem for quadratic cost consists of finding the most efficient way to take $\mu$ to $\nu$ given that the transportation cost to move from a point $x$ to a point $y$ is $|x-y|^{2}$. Hence, one is led to minimize
\begin{equation*}
\label{eqn:MOT}
\cost(T):= \int_{\R^{n}} |x-T(x)|^{2} \, d\mu(x)
\end{equation*}
among all maps $T$ such that $T_{\#}\mu = \nu$. A relaxed formulation of Monge's problem, due to Kantorovich, is to minimize
\begin{equation*}
\label{eqn:KOT}
\int_{\R^{n} \times \R^{n}} |x-y|^{2} \, d\pi(x,y)
\end{equation*}
among all transport plans $\pi$, namely the measures $\pi \in \Prob(\R^{n} \times \R^{n})$ whose marginals are $\mu$ and $\nu$. By a classical theorem of Brenier \cite{b}, the existence and uniqueness of an optimal transport plan are guaranteed when $\mu$ is absolutely continuous and $\mu$ and $\nu$ have finite second moments. Additionally, the optimality of a transport plan $\pi$ is equivalent to $\pi = (\Id \times \nab \phi)_{\#}\mu$ where $\phi$ is a convex function, often called the potential associated to the optimal transport. As a consequence, it follows that in the Monge problem, unique optimal maps exist as gradients of convex functions.
\begin{theorem}
\label{thm:Brenier}
Let $\mu, \nu \in \Prob(\R^{n})$ such that $\mu = f(x) \, dx$ and
\[ \int_{\R^{n}} |x|^{2} \, d\mu(x) + \int_{\R^{n}} |y|^{2} \, d\nu(y) < \infty. \]
Then, there exists a unique (up to sets of $\mu$-measure zero) optimal transport $T$ taking $\mu$ to $\nu$. Moreover, there is a convex function $\phi : \R^{n} \to \R$ such that $T = \nab \phi$.
\end{theorem} 

A direct consequence of Brenier's characterization of optimal transports as gradients of convex functions is that
\begin{equation}
\label{eqn:mono}
\langle x-y, T(x)-T(y) \rangle \geq 0 \qquad \text{for a.e. } x,y \in \R^{n},
\end{equation}
which follows immediately from the monotonicity of gradients of convex functions.

Suppose now that $\mu = f(x) \, dx$ and $\nu = g(y) \, dy$, and let $\phi$ be a convex function such that $T = \nab \phi$ for $T$ the optimal transport that takes $\mu$ to $\nu$. Assuming that $T = \nab \phi$ is a smooth diffeomorphism, the standard change of variables formula implies that
\[ f(x) = g(T(x))\det \nab T(x). \]
Hence, assuming that $g > 0$, we see that $\phi$ is a solution to the Monge-Amp\`{e}re equation
\[ \det D^{2}\phi = \frac{f}{g \circ \nab \phi}. \]

This formal link between optimal transportation and Monge-Amp\`{e}re (since, to deduce the above equation, we assumed that $T$ was already smooth) is at the heart of the regularity of optimal transport maps (see, for instance, \cite{dpf}
for more details).
In particular, Caffarelli showed the following in \cite{c1}
(see also \cite[Theorem 4.5.2]{fbook}):
\begin{theorem}
\label{thm:Caff-reg}
Let $X,\,Y \subset \R^{n}$ be bounded open sets, and $f : X \to \R^{+}$ and $g : Y \to \R^{+}$ be probability densities locally bounded away from zero and infinity.
If $Y$ is convex, then for any set $X'\subset \subset X$, the optimal transport $T = \nab \phi : X \to Y$ between $f(x)\, dx$ and $g(y)\, dy$ is of class $C^{0,\a}(X')$ for some $\a > 0$.
In addition, if $f \in C^{k,\b}_{\rm{loc}}(X)$ and $g \in C^{k,\b}_{\rm{loc}}(Y)$ for some $k\in \N \cup\{ 0\}$ and $\beta \in (0,1)$, then $\phi \in C^{k+2,\b}_{\rm{loc}}(X)$.
\end{theorem} 

As mentioned in \cite{adm}, Caffarelli's regularity result on optimal transports can be extended to the case where $f$ and $g$ are defined on all of $\R^{n}$ and assumed to be locally bounded away from zero and infinity.
Lastly, we note that optimal transport maps are stable under approximation (see \cite{v}). In particular, let $f_{j}$ and $g_{j}$ be locally uniformly bounded probability densities such that $f_{j} \to f$ and $g_{j} \to g$ in $L^{1}_{\rm{loc}}$. Then, the associated potentials $\phi_{j} \to \phi$ locally uniformly and $\nab\phi_{j} \to \nab\phi$ in measure.

We fix the following additional notation:

\[ \begin{array}{ll}
B_{R} &\text{ball of radius $R$ centered at the origin} \\
\mathcal{L}^{n} &n\text{-dimensional Lebesgue measure} \\
\H^{d} &d\text{-dimensional Hausdorff measure} \\
\mathbb{S}^{n-1} &\text{unit sphere in }\R^{n} \\
\omega_{n} &n\text{-dimensional Lebesgue measure of }B_{1} \subset \R^{n} \\
\end{array} \]

\section{Lipschitz Changes of Variables between Log-concave Measures}
\label{sec:Caffarelli}

We begin with two useful results of Caffarelli (see \cite{c2}).
They provide some motivation, and we briefly recall their proofs both for completeness and because we shall need them later.

\begin{lemma}
\label{lemma:finite max (Gen)}
Let $\mu = f(x) \, dx,\, \nu = g(x) \, dx \in \Prob(\R^{n})$ with finite second moments and $\nab \phi = T$ be the optimal transport taking $\mu$ to $\nu$.
Assume that $\log f \in L^\infty_{\rm loc}(\R^n)$ and that $g$ is bounded away from zero in the ball $B_{\j}$ for some $\j > 0$ and vanishes outside $B_{\j}$. 
Then,
$$
T(x) \to \j \frac{x}{|x|}\qquad \text{uniformly as }|x| \to\infty.
$$
In particular, for any fixed $\ep > 0$ and for all $\a \in \mathbb{S}^{n-1}$, the function $\phi(x+\e\a)+\phi(x-\e\a)-2\phi(x) \to 0$ as $|x| \to \infty$.
\end{lemma}

\begin{proof}
We begin by noticing that, as a consequence of Theorem \ref{thm:Caff-reg}, $T$ 
is continuous on $\R^n$ and, in particular,
the map $T$ is well defined at every point.

Let $x_{0} \in \R^{n}$ and $\t \in (0, \pi/4)$ be fixed, and consider the cone with vertex at $T(x_{0})$ and pointing in the $x_{0}$-direction
\[ \Gamma := \bigg{\{} y \in \R^{n} : \angle (x_{0}, y-T(x_{0})) \leq \frac{\pi}{2} - \t \bigg{\}}. \] 
By \eqref{eqn:mono} we see that
$$
\angle (x-x_{0}, T(x)-T(x_{0})) \leq \frac{\pi}2,
$$
hence
\[ \angle (x-x_0,x_0) \leq \angle (x-x_{0}, T(x)-T(x_{0})) + \angle (x_{0}, T(x)-T(x_{0}))
\leq \pi - \t  \qquad \forall \,x \text{ s.t. } \,T(x) \in \Gamma, \]
and so, up to a set of measure zero, the preimage of $\Gamma$ under $T$ is contained in the (concave) cone
\[ \Omega := \{ x \in \R^{n} : \angle(x_{0}, x-x_{0}) \leq \pi - \t \}. \]
Moreover, since $T_{\#}\mu = \nu$, 
\[ \inf_{x \in B_{\j}}g(x)\, \mathcal{L}^{n}(\Gamma \cap B_{\j}) \leq \nu(\Gamma \cap B_{\j}) = \nu(\Gamma) \leq \mu(\Omega). \]
Let $B = B_{(|x_{0}|\tan{\t})/2}$, and notice that $\Omega \subseteq \R^{n} \setminus B$. This proves that $\mu(\Omega) \leq \mu(\R^{n} \setminus B)$.

Now, $\mu(\R^{n} \setminus B) \to 0$ as $|x_{0}| \to \infty$ since $B$ covers $\R^{n}$ as $|x_{0}| \to \infty$. Recalling that $g$ is bounded away from zero in $B_{\j}$, we have that
\[ \lim_{|x_{0}| \to \infty} \mathcal{L}^{n}(\Gamma \cap B_{\j}) = 0. \]
Letting $\t \to 0$, we see that $T(x_{0}) \to \j \frac{x_{0}}{|x_{0}|}$. As the point $x_{0}$ was fixed arbitrarily, $\nab \phi(x) = T(x) \to \j \frac{x}{|x|}$ uniformly as $|x| \to \infty$. Thus, $\phi$ behaves like the cone $\j|x|$ at infinity. In particular, for any fixed $\ep > 0$ and for all $\a \in \mathbb{S}^{n-1}$, the function $\phi(x+\e\a)+\phi(x-\e\a)-2\phi(x) \to 0$ as $|x| \to \infty$.
\end{proof}

Thanks to Lemma~\ref{lemma:finite max (Gen)}, in \cite{c2,c2.1}, Caffarelli proved the following result.

\begin{theorem}
\label{thm:Caff}
Let $V,\, W \in C^{1,1}_{\rm{loc}}(\R^{n})$ be such that $e^{-V(x)} \, dx,\, e^{-W(x)} \, dx \in \Prob(\R^{n})$. Suppose there exist constants $0 < \l_{W},\, \L_{V} < \infty$ such that $D^{2}V(x) \leq \L_{V}\Id$ and $\l_{W}\Id \leq D^{2}W(x)$ for a.e. $x \in \R^{n}$. Then, the optimal transport $T$ that takes $e^{-V(x)} \, dx$ to $e^{-W(x)} \, dx$ is globally Lipschitz and satisfies
\begin{equation}
\label{eqn:gradT}
\| \nab T \|_{L^\infty(\R^{n})} \leq \sqrt{\L_{V}/\l_{W}}.
\end{equation}
\end{theorem}

\begin{proof}
By the stability of optimal transports, we may assume that $W$ is equal to infinity outside the ball $B_{\j}$ for some fixed $\j > 0$. Indeed, define
\[ W^{\j} := \begin{cases}
W &\text{ in } B_{\j} \\
\infty &\text{ in }  \R^{n} \setminus B_{\j}
\end{cases}\]
and $c_{\j} \in (0, \infty)$ such that
\[ \int_{\R^{n}} e^{c_{\j}-W^{\j}(x)} \, dx = 1. \]
Clearly, $e^{c_{\j}-W^{\j}} \to e^{-W}$ in $L^{1}(\R^{n})$ as $\j \to \infty$. Hence, if we prove \eqref{eqn:gradT} for the optimal transport $T^{\j}$ that takes $e^{-V(x)} \, dx$ to $e^{c_{\j}-W^{\j}(x)} \,dx$, letting $\j \to \infty$ we obtain the same estimate for $T$.

Also, by Theorem \ref{thm:Caff-reg}, the convex potential $\phi:\R^n\to \R$ associated to the optimal transport $T$
is of class $C^3$; therefore, $\phi$ satisfies the Monge-Amp\`{e}re equation 
\[ \det{D^{2}}\phi(x) = \frac{e^{-V(x)}}{e^{-W(\nab \phi(x))}}, \]
or equivalently,
\begin{equation}
\label{eqn:log-MA} 
\log \det{D^{2}}\phi(x) = -V(x) +  W(\nab \phi(x)).
\end{equation}
For fixed $\ep > 0$, we define the incremental quotient of a function $f:\R^n \to \R$ at $(x, \a) \in \R^{n} \times \mathbb{S}^{n-1}$ by
\[ f^{\ep}(x, \a) := f(x+\ep\a) + f(x-\ep\a) - 2f(x). \] 

By convexity of $\phi$ we see that $\phi^\ep \geq 0$.
Also, it follows by Lemma~\ref{lemma:finite max (Gen)} that $\phi^{\ep}\to 0$ as $|x| \to \infty$.
Thus $\phi^{\ep}$ attains a global maximum at some $(x_{0}, \a_{0}) \in \R^{n} \times \mathbb{S}^{n-1}$. Up to a rotation, we assume that $\a_{0} = \eone$. Thus,
\begin{equation}
\label{eqn:grad of diff quo} 
0 = \nab \phi^{\ep}(x_{0}, \eone) = \nab \phi(x_{0}+\ep\eone) + \nab \phi(x_{0}-\ep\eone) - 2\nab \phi(x_{0}).
\end{equation}
Moreover, because $\eone$ is the maximal direction,
\[ 0 = \partial_{\b}\phi^{\ep}(x_{0}, \eone) =\ep \langle \nab \phi(x_{0}+\ep\eone) - \nab \phi(x_{0}-\ep\eone), \b \rangle \qquad \forall \,\b \perp \eone. \]
Taking $\b = \text{e}_{i}$ for $i \neq 1$ and utilizing \eqref{eqn:grad of diff quo}, we see that all the components but the first of $\nab \phi(x_{0}+\ep\eone)$, $\nab \phi(x_{0}-\ep\eone)$, and $\nab \phi(x_{0})$ are equal. Let $\de := \langle \nab \phi(x_{0}+\ep\eone) - \nab \phi(x_{0}-\ep\eone), \eone \rangle/2$, and observe that, by \eqref{eqn:grad of diff quo},
\[\begin{split} 
\langle \nab \phi(x_{0}), \eone \rangle \pm \de 
&= \frac{1}{2}\langle \nab \phi(x_{0}+\ep\eone) + \nab \phi(x_{0}-\ep\eone), \eone \rangle \pm \frac{1}{2}\langle \nab \phi(x_{0}+\ep\eone) - \nab \phi(x_{0}-\ep\eone), \eone \rangle \\
&=\langle \nab \phi(x_{0}\pm\ep\eone), \eone \rangle. 
\end{split} \]
Hence, we conclude that
\begin{equation}
\label{eqn:1st eps-del est}
\nab \phi(x_{0} \pm \ep\eone) = \nab \phi(x_{0}) \pm \de \eone.
\end{equation}
Another consequence of $\phi^\ep$ achieving a maximum at $x_0$ is
\begin{equation}
\label{eqn:discrete max}
D^{2}\phi(x_{0}+\ep\eone) + D^{2}\phi(x_{0}-\ep\eone) - 2D^{2}\phi(x_{0}) \leq 0.
\end{equation} 
We recall that \begin{equation} 
\label{eqn:der of det}
\lim_{\ep \to 0^{+}}  \frac{\det{(A+\ep B)} - \det(A)}{\ep} = \det{(A)}\trace{(A^{-1}B)}
\end{equation}
for all square matrices $A$ and $B$ with $A$ invertible. Also, if we set $F(A) := \log\det A$, since $F$ is concave on the space of positive semidefinite $n \times n$ matrices and recalling \eqref{eqn:der of det}, we have
\[ \nab F(D^{2}\phi(x_{0})) = (D^{2}\phi(x_{0}))^{-1} \]
and
\[ F(D^{2}\phi(x_{0}\pm \ep\eone)) \leq F(D^{2}\phi(x_{0})) + \langle (D^{2}\phi(x_{0}))^{-1}, D^{2}\phi(x_{0}\pm \ep\eone) - D^{2}\phi(x_{0}) \rangle. \]
In particular, from \eqref{eqn:discrete max} and the convexity of $\phi$, we deduce that
\[ F(D^{2}\phi(x_{0}+\ep\eone)) + F(D^{2}\phi(x_{0}-\ep\eone)) - 2F(D^{2}\phi(x_{0})) \leq 0. \]

Now, let us, for fixed $\ep > 0$, consider the incremental quotient of \eqref{eqn:log-MA} at $(x_{0}, \eone)$. 
Using \eqref{eqn:1st eps-del est}, we realize that
\begin{equation}
\label{eqn:log-con}
V^{\ep}(x_{0}, \eone) \geq W^{\de}(\nab \phi(x_{0}), \eone).
\end{equation}
Observe that
\[ V^{\ep}(x_{0}, \eone) = \int_{0}^{\ep} \bigg{(} \int_{-t}^{t} \langle D^{2}V(x_{0} + s\eone)\eone, \eone \rangle \,ds \bigg{)} \,dt; \]
hence,
\begin{equation}
\label{V-eps est}  
V^{\ep}(x_{0}, \eone) \leq \L_{V}\ep^{2}. 
\end{equation}
Furthermore, from \eqref{eqn:1st eps-del est}, we similarly see that
\[ \l_{W}\de^{2} \leq W^{\de}(\nab \phi(x_{0}), \eone). \]
Combining this estimate with \eqref{V-eps est} and \eqref{eqn:log-con}, we get
\begin{equation}
\label{eqn:2nd eps-del est}
\ep\sqrt{\L_{V}/\l_{W}} \geq \de.
\end{equation}
Set $C := \sqrt{\L_{V}/\l_{W}}$. Since
\[ \phi^{\ep}(x_{0}, \eone) = \int_{0}^{\ep} \langle \nab \phi(x_{0}+t\eone) - \nab \phi(x_{0}-t\eone), \eone \rangle \, dt,\]
the convexity of $\phi$, \eqref{eqn:1st eps-del est}, and \eqref{eqn:2nd eps-del est} give us that
\[ \phi^{\ep}(x_{0}, \eone) = 2\de \ep \leq 2C\ep^{2}, \]
and so
\[ \| \nab T \|_{L^\infty(\R^n)}  = \| D^{2}\phi \|_{L^\infty(\R^n)} \leq 2C. \]
Notice that this is the desired estimate up to a factor $2$. We use a bootstrapping argument to remove this factor.
Suppose that $0 \leq \| D^{2}\phi \|_{L^\infty(\R^n)} \leq a_{0}$ for some $a_{0} > C$. For any $0 \leq t \leq \ep$, by \eqref{eqn:1st eps-del est} and \eqref{eqn:2nd eps-del est},
\[ |\langle \nab \phi(x_{0}+t\eone) - \nab \phi(x_{0}-t\eone), \eone \rangle| \leq \min \{ 2\ep C, 2a_{0}t \}. \] 
Thus,
\begin{equation*}
\phi^{\ep}(x_{0}, \eone) \leq \int_{0}^{\frac{\ep C}{a_{0}}} 2a_{0}t \, dt + \int_{\frac{\ep C}{a_{0}}}^{\e} 2\ep C \, dt = \ep^{2}\frac{(2C a_{0}-C^{2})}{a_{0}}.
\end{equation*}
In other words, if $\|D^{2}\phi\|_{L^\infty(\R^{n})} \leq a_{0}$ with $a_{0} > C$, then 
\[ \|D^{2}\phi\|_{L^\infty(\R^{n})} \leq \frac{(2C a_{0}-C^{2})}{a_{0}}. \] 
Starting with $a_{0} = 2C$ and repeating the above procedure an infinite number of times, we prove \eqref{eqn:gradT} since $C$ uniquely solves $(2C a-C^{2})/a = a$.
\end{proof}

\begin{remark}
\label{rmk:local}
Notice that the above proof relies only on the local behavior of our densities $e^{-V}$ and $e^{-W}$. In particular, the bounds on the Hessians of $V$ and $W$ are only used near the maximum point $x_{0}$ and its image $\nab \phi(x_{0})$, respectively. This simple observation will play an important role in the proof of Theorem~\ref{thm:main}.
\end{remark}

\begin{remark}
\label{rmk: Caff no ideal}
The above result is not ideal. Indeed, if $V = W$, then $T = \Id$ and one would like to have the bound $\|\nab T\|_{L^{\infty}(\R^{n})} \leq 1$ instead of $\|\nab T\|_{L^\infty(\R^{n})} \leq \sqrt{\L_{V}/\l_{V}}$.
\end{remark}

\section{Compactly Supported Perturbations: Proof of Theorem~\ref{thm:main}}
\label{sec:Compactly Supported Perturbations}

In the following lemma, we prove an upper bound on how far points travel under the transport map when the source measure is perturbed in a certain fixed ball $B_P$. 
We capture and quantify that our perturbations are compactly supported.
Lemma~\ref{lemma:R-image-T-est-V-cut} will be applied in the proof of Theorem~\ref{thm:main} to the inverse transport.

Furthermore, given our convex function $V$, we consider, for $\j \in \N$, 
\begin{equation}
\label{eqn:vj}
V^{\j} := \begin{cases}
V &\text{ in } B_{\j} \\
\infty &\text{ in } \R^{n} \setminus B_{\j},
\end{cases}
\end{equation}
and we approximate $e^{-V(x)} \, dx$ with compactly supported measures $e^{c_\j-V^{\j}(x)} \, dx$. 
This approximation is in the spirit of Caffarelli's approximation in the proof of Theorem~\ref{thm:Caff}. It allows us to find maximum points of a suitable function and guarantees that they do not escape to infinity in the proof of Theorem~\ref{thm:main}. This approximation procedure is purely technical. Hence, on a first reading of Lemma~\ref{lemma:R-image-T-est-V-cut}, the reader may just take $\j=\infty$.

\begin{lemma}
\label{lemma:R-image-T-est-V-cut}
Let $V \in C^{\infty}(\R^{n})$ be such that $\mu := e^{-V(x)} \, dx \in \Prob(\R^{n})$. 
Suppose that $V(0) = \inf_{\R^{n}} V$ and there exist constants $0 < \l,\, \L < \infty$ such that $\l\Id \leq D^{2}V(x) \leq \L\Id$ for all $x \in \R^{n}$.
Moreover, let $P > 0$, $p \in C^{\infty}_{c}(B_{P})$, and $c_p \in \R$ be such that $e^{-V(x)+c_p-p(x)} \, dx \in \Prob(\R^{n})$. 
Given $\j > P$, set $V^{\j}$ as in \eqref{eqn:vj}
and choose $c_{p,\j} \in (0, \infty)$ such that $\mu_{p,\j} :=e^{c_{p,\j}-V^{\j}(x)+c_p-p(x)} \, dx \in \Prob(\R^{n})$.
If $T$ is the optimal transport map that takes $\mu_{p,\j}$ to $\mu$, then there exist constants $P' = P'(P,\l,\L,\| p \|_{L^\infty(\R^n)}) > 0$ and $\j' = \j'(n, V(0),P,\l,\L,\| p \|_{L^\infty(\R^n)}) > P$ such that for all $\j \in [\j', \infty]$,
\begin{equation}
\label{eqn:image-T-V-cut}
T(B_{P}) \subseteq  B_{P'}.
\end{equation}
\end{lemma}

Even though this lemma is not independent of dimension as written (specifically, $\j'$ depends on $n$), the dimensional dependence does not affect the constant $P'$ and disappears in the limit as $\j \to \infty$. 
Thus, we can indeed prove a global estimate on the optimal transport taking $e^{-V(x)} \, dx$ to $e^{-V(x)+c_q-q(x)} \, dx$ that is independent of dimension.

Lemma~\ref{lemma:R-image-T-est-V-cut} is written under slightly different assumptions than Theorem~\ref{thm:main}.
In particular, besides the obvious additional regularity assumptions on $V$ and its perturbation, made only for simplicity, we have not required that the perturbation be semiconvex.
That said, if we assume the the distributional Hessian of $p$ is indeed bounded below by $-\l_{p}\Id$, then we can replace the dependence on $\| p \|_{L^\infty(\R^n)}$ with a dependence on $\lambda_p$, as explained in the following remark.

\begin{remark}
\label{rmk:l-infty-q}
Let $p$ be a function compactly supported in $B_P$ that satisfies the semiconvexity condition $D^{2} p \geq -\l_{p}\Id$ in the sense of distributions. 
Then, its $L^\infty$-norm is controlled by a constant depending only on $P$ and $\l_p$ (in particular, it is independent of dimension):
\begin{equation}
\label{eqn:l-infty-q}
\| p\|_{L^\infty(\R^n)} \leq 4 \lambda_p P^2.
\end{equation}
First, up to convolving $p$ with a standard convolution kernel, we can assume that $p$ is smooth.
Then, we observe that every 1-dimensional restriction $f_\alpha(t) = p (t\alpha)$, for $t\in \R$ and $\alpha\in \mathbb{S}^{n-1}$, is compactly supported in $[-P,P]$ and has second derivative bounded below by $-\lambda_p$.
This implies that 
\begin{equation}
\label{eqn:l-infty-q2}
\|f_\alpha'\|_{L^\infty(\R)} \leq 2\lambda_p P.
\end{equation}
Indeed, suppose to the contrary that $f_\alpha'(t_0)>2\lambda_p P$ for some $t_0 \in [-P,P]$. 
By integration, we would get
$$
0=f_\alpha'(P)\geq f_\alpha'(t_0)+\int_{t_0}^P f_\alpha'' (\tau)\,d\tau
>2\lambda_p P +\lambda_p(P-t_0)>0.
$$
Impossible. This proves \eqref{eqn:l-infty-q2}, and  
\eqref{eqn:l-infty-q} holds by integrating.
\end{remark}

Before proceeding with the proof of Lemma~\ref{lemma:R-image-T-est-V-cut}, we recall a Talagrand-type transport inequality. 
Given $\mu_1, \mu_2 \in \Prob(\R^n)$, we denote the squared Wasserstein distance between $\mu_1$ and $\mu_2$ by $W_2^2(\mu_1,\mu_2)$ (see \cite[Chapter 6]{v} for the general definition), and we consider their relative entropy
\begin{equation*}
\ent (\mu_2 | \mu_1) :=
\begin{cases}
\displaystyle{\int_{\R^n} \log \bigg( \frac{d\mu_2}{d\mu_1}\bigg) \, d\mu_2} &\text{if } \mu_2 \ll \mu_1
\\
\infty &\text{otherwise}.
\end{cases}
\end{equation*}
Here, $d\mu_2/d\mu_1$ is the relative density of $\mu_2$ with respect to $\mu_1$.
If $\mu_1 = e^{-V(x)} \, dx$ for some $V \in C^{2}(\R^n)$ such that $D^{2}V(x) \geq \lambda_V \Id$ for all $x \in \R^n$, we have that (see \cite{c-e}, applied in the particular case when $\mu_1$ and $\mu_2$ are probability measures) 
\begin{equation}
\label{eqn:tala}
W_2^2 (\mu_1,\mu_2) \leq \frac{2}{\lambda_V}\ent (\mu_2 | \mu_1).
\end{equation}

In our applications, $W_2^2(\mu_1,\mu_2)$ coincides with the cost of the optimal transport taking $\mu_2$ to $\mu_1$.

\begin{proof}[Proof of Lemma~\ref{lemma:R-image-T-est-V-cut}]
Notice first that, as a consequence of Theorem \ref{thm:Caff-reg}, $T$
is continuous.

Assume there exists a point $x_{0} \in B_{P}$ with $T(x_{0}) \notin B_{10P}$ (otherwise, the statement is true with $P' = 10P$). 
We show that $T(x_{0}) \in B_{P'}$ for some $P' = P'(P,\l,\L,\| p \|_{L^\infty(\R^n)}) > 0$ that will be chosen later. 
Let 
\[
\bar{x} := x_{0}+3P\,\frac{T(x_{0})-x_{0}}{|T(x_{0})-x_{0}|},
\]
and define the constant $C_{0}$ and ball $B$ by $C_{0}P = |T(x_{0})-x_{0}|$ and $B := B_{P}(\bar{x})$.
Also, let $F : B \rightarrow \R^{n}$ be the projection of a point $y \in B$ onto the hyperplane through $T(x_{0})$ and perpendicular to $y-x_{0}$. The map $F$ is well-defined because $x_{0} \notin B$ (see Figure~\ref{fig:ballsentfar}). 
Let us assume that $\j'>6P$, so that $B \subseteq B_\j$.

\begin{figure}[ht]
\center \includegraphics[scale=1]{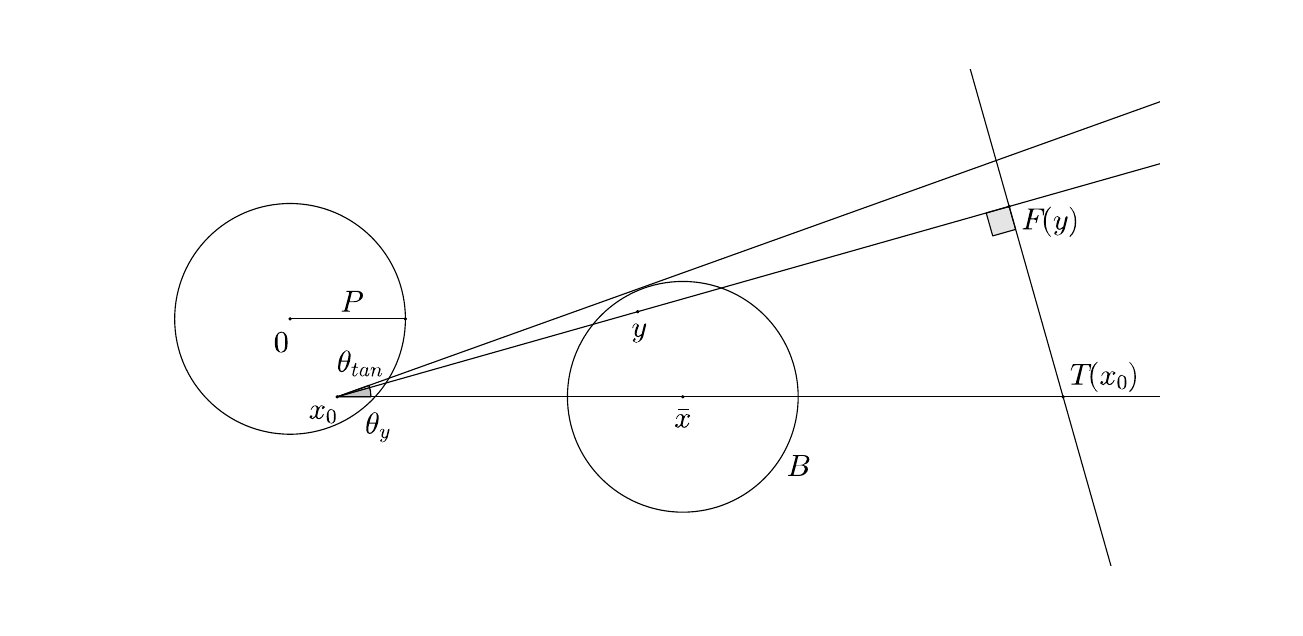}
\caption{The optimal transport sends $B$ far away.}
\label{fig:ballsentfar}
\end{figure} 
By \eqref{eqn:mono}, we have that
\[ \langle y-x_{0}, T(y)-T(x_{0}) \rangle \geq 0 \qquad \forall \,y \in B, \]
and as $F(y)$ is the closest point to $y$ in the set $\{z \in \R^{n} : \langle y-x_{0}, z-T(x_{0}) \rangle \geq 0$\},
\[ |T(y)-x_{0}| \geq |F(y)-x_{0}| \qquad \forall \,y \in B \]
(see Figure~\ref{fig:ballsentfar}). 
Given any $y \in B$, either $x_{0}$, $y$, and $\bar{x}$ determine a plane, call it $\Gamma_{y}$, within which $x_{0}$, $F(y)$, and $T(x_{0})$ determine a right triangle, or $x_{0}$, $y$, and $\bar{x}$ are collinear. 
Thus, 
\[ |F(y)-x_{0}| = C_{0}P\cos\t_{y} \] 
where $\t_{y}$ is the angle between $F(y)-x_{0}$ and $T(x_{0})-x_{0}$. 
Now, $\Gamma_{y} \cap \partial B$ is a circle of radius $P$ centered at $\bar{x}$. 
Letting $\t_{tan}$ be the angle between the line through $x_{0}$ and tangent to $\Gamma_{y} \cap \partial B$ and the line through $T(x_{0})$ and $x_{0}$, we see that $\t_{y} \leq \t_{tan}$. 
(While there are two such tangent lines, the angles they determine with the line through $T(x_{0})$ and $x_{0}$ are the same. 
Again, see Figure~\ref{fig:ballsentfar}.) 
Moreover, $|x_{0}-\bar{x}|=3P$ and $\cos\t_{tan} = 2\sqrt{2}/3$. 
Consequently,
\[ |F(y)-x_{0}| \geq C_{0}P\,\cos\t_{tan} \geq \frac{C_{0}2\sqrt{2}P}{3} \] 
and
\[ |T(y)-y| \geq |T(y)-x_{0}| - |y-x_{0}| > \frac{C_{0}2\sqrt{2}P}{3}-4P \qquad \forall \,y \in B. \]
Since $V(0) = \inf_{\R^{n}} V(x)$ and $\l\Id \leq D^{2}V(x) \leq \L\Id$, by restricting $V$ to $1$-dimensional lines through the origin we have that 
\begin{equation}
\label{poly. approx.}
V(0) + \frac{\l}{2}|x|^{2} \leq V(x) \leq V(0) + \frac{\L}{2}|x|^{2} \qquad \forall \,x \in \R^{n};
\end{equation}
hence, as $B \subseteq B_{6P}$,
\[ V(x) \leq V(0) + 18\L P^{2} \qquad \forall \,x \in B. \]

We now estimate $\cost(T)$.
Since $B_{P} \cap B = \emptyset$ and $B \subseteq B_{\j}$, we have
\begin{equation}
\label{eqn:cost-below-V-cut}
\cost(T) \geq \int_{B} |T(x)-x|^{2}e^{c_{p,\j}-V^{\j}(x)+c_{p}} \, dx \geq \Bigg{[}\frac{C_{0}2\sqrt{2}P}{3}-4P\Bigg{]}^{2}e^{c_{p,\j}-V(0)-18\L P^{2}+c_{p}} \mathcal{L}^{n}(B_{P}).
\end{equation}
Furthermore, we claim that the following upper bound on $\cost(T)$ holds:
\begin{equation}
\label{eqn:costT above}
\cost(T) \leq \frac{6}{\lambda}\|p\|_{L^\infty(\R^n)}e^{c_{p,\j}+c_{p}+\|p\|_{L^\infty(\R^n)}}\mu(B_P).
\end{equation}
To see this, first, apply the Talagrand-type transport inequality \eqref{eqn:tala} with $\mu_1 = \mu$ and $\mu_2 = \mu_{p,\j}$ to find that
\begin{equation}
\label{eqn:tala-applied}
\cost(T) \leq \frac{2}{\lambda} \int_{\R^n} (c_{p,\j}+c_p-p(x))e^{c_{p,\j}-V^{\j}(x)+c_{p}-p(x)} \, dx.
\end{equation}
Second, choose $\j'>6P$, so that 
$$
\int_{\R^n \setminus B_{\j'}} e^{-V(0) - \frac{\lambda}{2}|x|^2+\|p\|_{L^\infty(\R^n)}} \, dx \leq 1- \exp\bigg(-\| p \|_{L^\infty(\R^n)}\int_{B_P}e^{-V(0)-\frac{\Lambda}{2}|x|^2} \, dx \bigg).
$$
Notice that $|c_p| \leq \|p\|_{L^\infty(\R^n)}$ since
\begin{equation}
\label{eqn:c_p}
e^{-c_p} = \int_{\R^n} e^{-p(x)} \, d\mu(x).
\end{equation}
So, for every $\j \geq \j'$, observe that
\begin{equation*}
\begin{split}
e^{-c_{p,\j}} &= \int_{B_\j} e^{-V(x)+c_p-p(x)} \, dx = 1- \int_{\R^n \setminus B_\j} e^{-V(x)+c_p} \, dx
\geq 1-\int_{\R^n \setminus B_{\j'}} e^{-V(0)-\frac{\lambda}{2}|x|^2+\|p\|_{L^\infty(\R^n)}} \, dx 
\\& \geq \exp\Big( -\| p \|_{L^\infty(\R^n)} \int_{B_P}e^{-V(0)-\frac{\Lambda}{2}|x|^2} \, dx \Big),
\end{split}
\end{equation*}
and then, recalling that $c_{p,\j}>0$, note
\begin{equation}
\label{eqn:c-j}
c_{p,\j} \leq\| p \|_{L^\infty(\R^n)}\int_{B_P}e^{-V(0)-\frac{\Lambda}{2}|x|^2} \, dx \leq \| p \|_{L^\infty(\R^n)} e^{c_{p,\j}+c_{p}+\| p \|_{L^\infty(\R^n)}}\mu(B_P).
\end{equation}
Now, use Jensen's inequality on \eqref{eqn:c_p} and that $p$ is supported in $B_P$ to deduce that
\begin{equation}
\label{eqn:Jensen-applied}
\begin{split}
c_{p} - \int_{B_P} p(x)\,e^{c_{p,\j}-V^{\j}(x)+c_{p}-p(x)} \, dx &\leq \int_{B_P} p(x)\,e^{c_{p,\j}+c_{p}}\Big[ e^{-c_{p,\j} - c_p} - e^{-p(x)} \Big] \, d\mu(x)  \\
&\leq 2\|p\|_{L^\infty(\R^n)}e^{c_{p,\j}+c_p+\|p\|_{L^\infty(\R^n)}}\mu(B_P).
\end{split}
\end{equation}
Finally, combine \eqref{eqn:tala-applied}, \eqref{eqn:c-j}, and \eqref{eqn:Jensen-applied} to see that \eqref{eqn:costT above} holds as claimed.

In particular, since $\mu (B_P) \leq e^{-V(0)}\mathcal{L}^n(B_P)$, we have that
\begin{equation}
\label{eqn:Brenier}
\cost(T) \leq \frac{6}{\lambda}\|p\|_{L^\infty(\R^n)}e^{c_{p,\j}-V(0)+c_{p}+\|p\|_{L^\infty(\R^n)}}\mathcal{L}^{n}(B_P),
\end{equation}
provided that $\j \geq \j'$. Thus, \eqref{eqn:cost-below-V-cut} and \eqref{eqn:Brenier} imply that
\[
C_{0} \leq C':= 3\sqrt{2} + \frac{9 e^{9\L P^{2}+\frac{\|p\|_{L^\infty(\R^n)}}{2}}}{2P}\Bigg[\frac{\| p \|_{L^\infty(\R^n)}}{\l}\Bigg]^{1/2}.
\]
This proves the existence of an upper bound on $C_{0}$ depending only on $P,\,\l,\,\L$ and $\| p \|_{L^\infty(\R^n)}$.

Taking $P' := (C' + 1)P$, we deduce that 
\[ |T(x_{0})| \leq |T(x_{0})-x_{0}| + |x_{0}| \leq C_{0}P + P \leq P', \]
which proves \eqref{eqn:image-T-V-cut}.
\end{proof}

The following result is a Pogorelov-type a priori estimate on pure second derivatives of the potential associated to our optimal transport.
This technique is inspired by Pogorelov's original argument for the classical Monge-Amp\`ere equation \cite{p}. In our case, we face the additional difficulty of constructing an auxiliary function $h$ that compensates for the concavity of our perturbation and the growth of our convex function at infinity. 
Assuming that our auxiliary function attains a finite maximum, we provide a quantitative estimate on the value of $h$ at its finite maximum. 
This result contains and overcomes the primary obstacles to demonstrating that our optimal transport is globally Lipschitz.

Before stating the result, we introduce some constants and an auxiliary function $\psi$, all depending only on the constants $R,\,\l,\,\L$, and $\l_{q}$ that appear in Theorem~\ref{thm:main}. 
Define the constants $P>0$ and $Q>0$ by 
\begin{equation} 
\label{defn:P}
P := \frac{2\l_{q}+4\l_{q}R}{\l} + 1 + R \qquad \text{and}\qquad Q := \frac{\l}{2\l_{q}}+1+R;
\end{equation}
let $\overline{\psi} \in C^{2}([0, \infty))$ be given by
\begin{equation}
\label{defn:phi}
\overline{\psi}(t) := \int_{0}^t\int_0^s \vartheta(r) \,dr\, ds, \quad
\vartheta(r) := \begin{cases}
\l_{q} &r \in [0, R] \\
-\l_{q}r+\l_{q}+\l_{q}R &r \in [R,Q] \\
\frac{\l_{q}\l^{2}r}{4\l_{q}^{2}+8\l_{q}^{2}R-\l^{2}} - \frac{2\l_{q}^{2}\l+4\l_{q}^{2}\l R+\l_{q}\l^{2}+\l_{q}\l^{2}R}{4\l_{q}^{2}+8\l_{q}^{2}R-\l^{2}} &r \in [Q, P]\\
0 &r\in [P, \infty);
\end{cases}
\end{equation}
and let $\psi \in C^{2}(\R^{n})$ be defined by 
\begin{equation}
\label{defn:psii}
\psi(y) := \overline{\psi}(|y|).
\end{equation}
Observe that the function $\overline{\psi}$ is defined in such a way that $\overline{\psi}''\geq -\l/2$ in $[0,\infty)$, $\overline{\psi}= \l_{q}|\cdot|^{2}/2$ on $[0,R]$, and $\overline{\psi}'$ is supported in $B_{P}$ (see Figure \ref{fig:overlinevarphi}).

\begin{figure}[ht]
\center \includegraphics[scale=1]{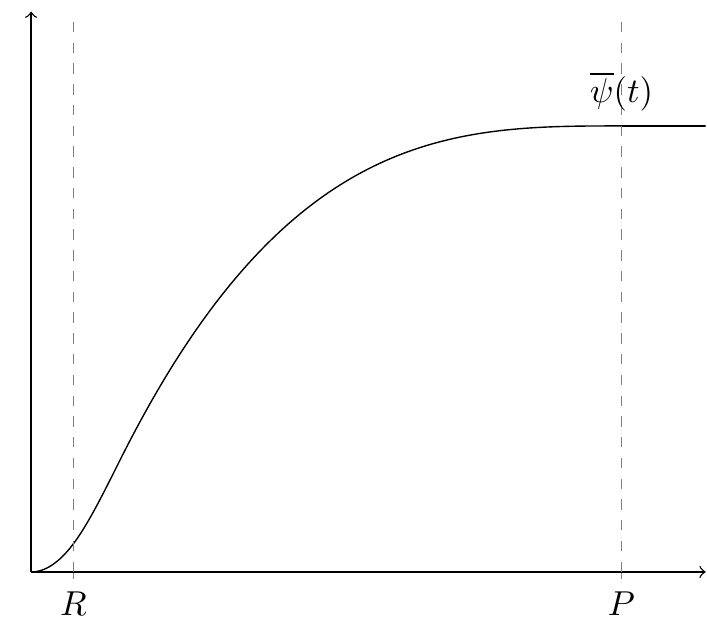}
 \caption{The graph of $\overline\psi$.}
\label{fig:overlinevarphi}
\end{figure}

\begin{proposition}
\label{prop:pogo}
Let $V,\,\l,\,\L,\,R,\,q,\,\l_{q}$, and $c_{q}$ be defined as in Theorem~\ref{thm:main}. Assume, additionally, that $V$ and $q$ are smooth. Let $P,\, \overline \psi$, and $\psi$ be defined as in \eqref{defn:P}, \eqref{defn:phi}, and \eqref{defn:psii}. Given $\j > P$, set $V^{\j}$ as in \eqref{eqn:vj} and choose $c_{q,\j} \in (0, \infty)$ such that $e^{c_{q,\j}-V^{\j}(x)+c_q-q(x)} \, dx \in \Prob(\R^{n})$. Also, let $\phi \in C^{\infty}(\R^{n})$ solve 
\[ \det{D^{2}\phi} = \frac{e^{-V}}{e^{c_{q,\j} - V^{\j}(\nab \phi) + c_{q} - q(\nab \phi)}}, \]
and assume that there exist constants $\j' , P' > 0$ such that for all $\j \in [\j' , \infty]$,
\begin{equation}
\label{eqn:R3-R2}
\nab \phi(\R^{n} \setminus B_{P'}) \subseteq \R^{n} \setminus B_{P},
\end{equation} 
or equivalently, that $[\nab \phi]^{-1}(B_{P}) \subseteq B_{P'}$. If 
\begin{equation}
\label{defn:hh}
h(x, \a) := \phi_{\a\a}(x)e^{\psi(\nab \phi(x))} 
\end{equation}
attains a maximum at some point $(x_{0},\a_{0})$ among all possible $(x,\a) \in \R^{n} \times \mathbb{S}^{n-1}$, then there exists a constant $C = C(R,P',\l,\L,\l_{q})>0$, yet independent of $n$, such that
\begin{equation*}
\label{eqn:Pogorelov (Gen)}
h(x_0,\a_0) \leq C.
\end{equation*}
\end{proposition}

\begin{proof}
Since, by assumption, $(x_0, \alpha_0)$ is a maximum point of $h$, we have $\sup_{|\a|=1} \phi_{\a\a}(x_0) = \phi_{\a_0\a_0}(x_0)$. 
This implies that $\a_0$ is an eigenvector of $D^2 \phi(x_0)$. 
Therefore, up to a rotation, we assume that $\a_{0} = \eone$ and that $D^{2}\phi$ is diagonal at $x_{0}$. Throughout this proof, the function $h$ is seen as a function of the variable $x$ with $\a_{0}$ fixed. Then, at $x_{0}$ we compute that
\begin{equation}
\label{eqn:1st der est (Gen)}
0 = (\log{h})_{i} = \frac{\phi_{11i}}{\phi_{11}} + \psi_{k}(\nab \phi)\phi_{ki},
\end{equation}
for all $1 \leq i \leq n$, and
\begin{equation}
\label{eqn:2nd der est (Gen)}
0 \geq \phi^{ij}(\log{h})_{ij} = \phi^{ij}\bigg{[}\frac{\phi_{11ij}}{\phi_{11}} - \frac{\phi_{11i}\phi_{11j}}{\phi_{11}^{2}} + \psi_{k}(\nab \phi)\phi_{kij} + \psi_{kl}(\nab \phi)\phi_{ik}\phi_{jl}\bigg{]}
\end{equation}
where we denote the inverse matrix of $(\phi_{ij})$ by $(\phi^{ij})$.

Let $\V^{\j} := V^{\j}-c_{q,\j}+q-c_{q}$. 
Using \eqref{eqn:der of det}, we differentiate the equation 
\begin{equation}
\label{eqn:MongeAmpere (Gen)}
\log{\det{D^{2}\phi}} = -V+\V^{\j}(\nab \phi)
\end{equation}
in the $\eone$-direction twice to obtain 
\[ \phi^{ij}\phi_{1ij} = -V_{1}+\V^{\j}_{i}(\nab \phi)\phi_{1i} \]
and
\begin{equation}
\label{eqn:2nd-deriv (Gen)}
\phi^{ij}\phi_{11ij} - \phi^{il}\phi^{kj}\phi_{1ij}\phi_{1kl} = -V_{11}+\V^{\j}_{i}(\nab \phi)\phi_{11i} + \V^{\j}_{ij}(\nab \phi)\phi_{1i}\phi_{1j}.
\end{equation}
By \eqref{eqn:2nd der est (Gen)} and \eqref{eqn:2nd-deriv (Gen)}, we deduce that at $x_0$
\begin{align}
\begin{split}
\label{eqn:pogo (Gen)}
0 \geq \phi^{il}\phi^{kj}\phi_{1ij}\phi_{1kl} &- V_{11} + \V^{\j}_{i}(\nab \phi)\phi_{11i} + \V^{\j}_{ij}(\nab \phi)\phi_{1i}\phi_{1j} \\
&- \frac{\phi^{ij}\phi_{11i}\phi_{11j}}{\phi_{11}} + \phi_{11}\phi^{ij}\psi_{k}(\nab \phi)\phi_{kij} + \phi_{11}\phi^{ij}\psi_{kl}(\nab \phi)\phi_{ik}\phi_{jl}.
\end{split}
\end{align}

We estimate each term in \eqref{eqn:pogo (Gen)} from below.
Recall that $(\phi_{ij})$ and $(\phi^{ij})$ are diagonal at $x_{0}$. Therefore, $\phi^{ii} = 1/\phi_{ii}$, and we see that
\[ \phi^{il}\phi^{kj}\phi_{1ij}\phi_{1kl} - \frac{\phi^{ij}\phi_{11i}\phi_{11j}}{\phi_{11}} = \sum_{i=1}^n \sum_{k=2}^n \phi^{ii}\phi^{kk}\phi_{1ik}^{2} \geq 0 \]
and
\[ \V^{\j}_{ij}(\nab \phi)\phi_{1i}\phi_{1j} = \V^{\j}_{11}(\nab \phi)\phi_{11}^{2}. \]
Because $h$ has a maximum at $\eone$ among all directions, 
\begin{equation}
\label{eqn:u11max-direct}
\phi_{11}(x_{0}) \geq \phi_{ii}(x_{0}),
\end{equation}
and so
\[ \phi_{11}\phi^{ij}\psi_{kl}(\nab \phi)\phi_{ik}\phi_{jl} = \phi_{11}\psi_{ii}(\nab \phi)\phi_{ii} \geq \psi_{ii}(\nab \phi)\phi_{ii}^{2}. \]
Additionally, differentiating \eqref{eqn:MongeAmpere (Gen)} in the $\text{e}_{k}$-direction, we have that 
\[ \phi^{ij}\phi_{kij} = -V_{k}+\V^{\j}_{i}(\nab \phi)\phi_{ki}. \]
By \eqref{eqn:1st der est (Gen)}, it then follows that 
\begin{align*}
\begin{split}
\V^{\j}_{i}(\nab \phi)\phi_{11i} + \phi_{11}\phi^{ij}\psi_{k}(\nab \phi)\phi_{kij} &= \V^{\j}_{i}(\nab \phi)\phi_{11i} + \psi_{k}(\nab \phi)(-V_{k}+\V^{\j}_{i}(\nab \phi)\phi_{ki})\phi_{11} \\
&= -\psi_{k}(\nab \phi)V_{k}\phi_{11},
\end{split}
\end{align*}
and, consequently, \eqref{eqn:pogo (Gen)} becomes
\begin{equation}
\label{eqn:pogo-with-work}
0 \geq  \V^{\j}_{11}(\nab \phi) \phi_{11}^{2} +\sum_{i = 1}^{n} \psi_{ii}(\nab \phi)\phi_{ii}^{2} -\psi_{k}(\nab \phi)V_{k}\phi_{11} -\L.
\end{equation}

If $x_{0} \in \R^n \setminus B_{P'}$, \eqref{eqn:R3-R2} implies that $\nab \phi(x_{0}) \in \R^{n} \setminus B_{P}$. Then, $\psi_{k}(\nab \phi)V_{k}\phi_{11} = 0$ since the gradient of $\psi$ is zero outside $B_{P}$ by construction. If, on the other hand, $x_{0} \in B_{P'}$, then 
\[ \psi_{k}(\nab \phi)V_{k}\phi_{11} \leq \L P' \|\nab \psi\|_{L^{\infty}(\R^{n})}\phi_{11}. \]
(Here, we have used that $V(0) = \inf_{\R^n} V$ and that $D^2 V \leq \L \Id$ to show $V_k$ is bounded above by $\L P'$.) In both cases, we deduce that
\[ \psi_{k}(\nab \phi)V_{k}\phi_{11} \leq \tilde{C}\phi_{11} \]
for a constant $\tilde{C}$ depending only on $R$, $P'$, $\l$, $\L$, and $\l_{q}$.
Thus, by \eqref{eqn:pogo-with-work}, we have that
\begin{equation}
\label{eqn:pogo-with-more-work}
0 \geq  \V^{\j}_{11}(\nab \phi) \phi_{11}^{2} +\sum_{i = 1}^{n} \psi_{ii}(\nab \phi)\phi_{ii}^{2} - \tilde{C}\phi_{11} - \L.
\end{equation}
We claim that
\begin{equation}
\label{eqn:piece-of-pogo-2}
\V^{\j}_{11}(\nab \phi) \phi_{11}^{2} +\sum_{i = 1}^{n} \psi_{ii}(\nab \phi)\phi_{ii}^{2} \geq \frac{\l}{2}\phi_{11}^{2}.
\end{equation}
Indeed, let us consider two cases, according to whether or not $\nab \phi (x_{0})$ belongs to $B_{R}$. If $\nab \phi(x_{0}) \in B_{R}$, then
\[ \V^{\j}_{11}(\nab \phi) \phi_{11}^{2} + \sum_{i = 1}^{n} \psi_{ii}(\nab \phi)\phi_{ii}^{2} \geq \l \phi_{11}^{2} - \l_{q}\phi_{11}^{2} + \l_{q} \phi_{11}^{2} = \l \phi_{11}^{2}, \]
and \eqref{eqn:piece-of-pogo-2} follows. In the case that $\nab \phi(x_{0}) \notin B_{R}$, we compute the derivatives of $\psi$ in terms of the derivatives of $\overline \psi$. Observe that
\[
\psi_{i}(y) = \frac{\overline{\psi}'(|y|)y_{i}}{|y|}
\qquad\text{and}\qquad
\psi_{ii}(y) = \overline{\psi}''(|y|)\frac{y_{i}^{2}}{|y|^{2}} + \frac{\overline{\psi}'(|y|)}{|y|}\bigg(1 -\frac{y_{i}^{2}}{|y|^{2}} \bigg).
\]
Thus,
\[ \psi_{ii}(\nab \phi) 
\geq \overline{\psi}''(|\nab \phi|)\frac{\phi_{i}^{2}}{|\nab \phi|^{2}} \geq -\frac{\l}{2}\frac{\phi_{i}^{2}}{|\nab \phi|^{2}} \]
since $\overline{\psi}'' \geq -\l/2$ in $[0,\infty)$ and $\overline{\psi}' \geq 0$.
Then, \eqref{eqn:u11max-direct} implies that
\begin{equation}
\label{eqn:phi11-outside}
\begin{split}
\sum_{i = 1}^{n} \psi_{ii}(\nab \phi)\phi_{ii}^{2}
\geq -\frac{\l}{2} \sum_{i = 1}^{n} \frac{\phi_{i}^{2}}{|\nab \phi|^{2}} \phi_{ii}^{2}
\geq -\frac{\l}{2} \phi_{11}^{2}\sum_{i = 1}^{n} \frac{\phi_{i}^{2}}{|\nab \phi|^{2}} = -\frac{\l}{2}\phi_{11}^{2}.
\end{split}
\end{equation}
As $\nab \phi(x_{0}) \notin B_{R}$, we know $\V_{11}^{\j}(\nab \phi(x_{0})) = V_{11}^{\j}(\nab \phi(x_{0}))$. It follows that
\begin{equation}
\label{eqn:W11-outside}
\V^{\j}_{11}(\nab \phi)\phi_{11}^{2} \geq \l\phi_{11}^{2}.
\end{equation}
By \eqref{eqn:phi11-outside} and \eqref{eqn:W11-outside}, we deduce that \eqref{eqn:piece-of-pogo-2} holds in this case as well. 

Combining \eqref{eqn:pogo-with-more-work} and \eqref{eqn:piece-of-pogo-2}, we observe that
\begin{equation}
\label{eqn:quadratic}
0 \geq  \frac{\l}{2}\phi_{11}^{2} - \tilde{C}\phi_{11} - \L.
\end{equation}
Solving the quadratic equation in \eqref{eqn:quadratic}, we find that
\[ \phi_{11}(x_0) \leq 
\frac{\tilde{C} + \sqrt{\tilde{C}^{2}+ 2\l\L}}{\l} \leq 2\tilde{C}/\l + \sqrt{2\L/\l}. \]
As $\psi$ is bounded in $\R^{n}$ by definition, it follows that
\[ h(x) \leq h(x_{0}) \leq \phi_{11}(x_{0})e^{\|\psi\|_{L^\infty(\R^{n})}} \leq C \]
for a constant $C$ depending on $R$, $P'$, $\l$, $\L$, and $\l_{q}$, yet independent of $n$, as desired.
\end{proof}

Notice that if $\l_{q} = 0$, then $\overline{\psi} = 0$. 
In this case, the constant $\tilde{C}$ found in the proof above is zero, and we recover the global Lipschitz constant obtained by Caffarelli in Theorem~\ref{thm:Caff} up to a factor of $\sqrt{2}$ (this is a better bound than the one provided by the proof of Theorem~\ref{thm:Caff} before the final bootstrapping argument).

\begin{proof}[Proof of Theorem~\ref{thm:main}]
We first prove the statement assuming that $V$ and $q$ are smooth.
For every $\j > R$ set $V^{\j}$ as in \eqref{eqn:vj},
and choose $c_{q,\j} \in (0, \infty)$ such that $e^{c_{q,\j}-V^{\j}(x)+c_q-q(x)} \, dx \in \Prob(\R^{n})$.
Let $T^{\j}$ be the optimal transport map that takes $e^{-V(x)} \, dx$ to $e^{c_{q,\j}-V^{\j}(x)+c_{q}-q(x)} \, dx$. Since the density $e^{c_{q,\j}-V^{\j}+c_{q}-q}$ is supported in a convex set,  smooth on its support, and is bounded from above and below by positive constants, by Theorem~\ref{thm:Caff-reg}, we deduce that $T^{\j} \in C^{\infty}(\R^{n})$. By the stability of optimal transport maps, it suffices to show that for all $\j \geq \j'$ ($\j'$ to be chosen possibly depending on $n$) we have that 
\begin{equation}
\label{eqn:ts-rho}
\| \nab T^{\j}\|_{L^\infty(\R^{n})} \leq C
\end{equation}
for some constant $C > 0$ depending only on $R$, $\l$, $\L$, and $\l_{q}$.

Let $P,\, \psi$, and $h$ be defined as in  \eqref{defn:P}, \eqref{defn:psii}, and \eqref{defn:hh}. Applying Lemma~\ref{lemma:R-image-T-est-V-cut} to the optimal transport $[T^{\j}]^{-1}$, we see that there exist constants $\j'$ and $P' = P'(R,\l,\L,\l_q) > 0$ (see Remark~\ref{rmk:l-infty-q}) such that $[T^{\j}]^{-1}(B_{P}) \subseteq B_{P'}$ for all $\j \in [\j', \infty]$; that is, letting $\nab \phi = T^{\j}$ (for simplicity we omit in $\phi$ the dependence on $j$, which can be any number greater than $j'$ in the following),
\begin{equation}
\label{eqn:R3-R2-applied}
\nab \phi(\R^{n} \setminus B_{P'}) \subseteq \R^{n} \setminus B_{P}.
\end{equation}

We split the proof in two cases, according whether or not $h$ achieves a maximum in $\Om = \R^{n} \times \mathbb{S}^{n-1}$. If there exists $(x_{0}, \a_{0}) \in \Om$ such that
\[ h(x_{0}, \a_{0}) = \sup_{\Om} h(x, \a), \]
then we apply Proposition~\ref{prop:pogo} and see that
\[ \sup_{\mathbb{S}^{n-1}} \|\phi_{\a\a}\|_{L^{\infty}(\R^{n})} \leq \|h\|_{L^\infty(\Om)} \leq C, \]
which proves \eqref{eqn:ts-rho}.

Otherwise, we consider the maxima of $h$ in $\Om_{m} :=\overline{B}_{m} \times \mathbb{S}^{n-1}$ with $m \in \N$. Let
\[ h (x_{m}, \a_{m}) =  \sup_{\Om_{m}} h(x, \a). \]
Notice that $h(x_{m}, \a_{m})$ is nondecreasing (and not definitively constant) and $|x_{m}| \uparrow \infty$ as $m \to \infty$. Now, consider the functions $h^{\ep}$ approximating $h$ defined by
\[ h^\ep(x, \a) := [\phi(x+\ep\a) + \phi(x-\ep\a) - 2\phi(x)]e^{\psi(\nab \phi(x))} \qquad \forall\,(x, \a) \in \Om. \]
Since $\phi$ is smooth, we know that $h^{\ep} \to h$ locally uniformly in $\Omega$ as $\ep \to 0$. Furthermore, by Lemma~\ref{lemma:finite max (Gen)}, 
\begin{equation}
\label{eqn:conv-h-eps}
\lim_{|x| \to \infty} h^{\ep}(x, \a) = 0
\end{equation}
uniformly with respect to $x$ and $\a$. Since $h^\ep \geq 0$ (by the convexity of $\phi$), the function $h^{\ep}(x, \a)$ has a finite maximum point $(x^{\ep}, \a^{\ep})$. 

We claim that for sufficiently small $\ep$ (possibly depending on $n$ and on the sequence $\{(x_{m}, \a_{m})\}_{m \in \N}$)
\begin{equation}
\label{eqn:max-outside-deg-eps}
x^{\ep} \notin B_{P'}.
\end{equation}
Indeed, let $m_{0}$ and $m_{1}$ be such that $x_{m_{0}} \notin B_{P'}$ and $h (x_{m_{1}}, \a_{m_{1}}) > h (x_{m_{0}}, \a_{m_{0}})$. 
Since $h^\ep$ converges to $h$ locally uniformly, there exists $\ep_{0} > 0$ such that
\[ |h^{\ep}(x, \a)-h(x, \a)| \leq \frac{h(x_{m_{1}}, \a_{m_{1}}) - h(x_{m_{0}}, \a_{m_{0}})}{4} \]
for every $x \in \overline{B}_{|x_{m_{1}}|+1}$, $\a \in \mathbb{S}^{n-1}$, and $\ep \leq \ep_{0}$.
So, for every $\ep\leq \ep_{0}$, we have that
\begin{equation}
\label{eqn:high-at-m-1}
h^{\ep}(x_{m_{1}}, \a_{m_{1}}) \geq h(x_{m_{1}}, \a_{m_{1}}) - |h^{\ep}(x_{m_{1}}, \a_{m_{1}}) - h(x_{m_{1}}, \a_{m_{1}})| \geq \frac{3h(x_{m_{1}}, \a_{m_{1}}) + h(x_{m_{0}}, \a_{m_{0}})}{4}.
\end{equation}
Thus,
\begin{equation}
\label{eqn:low-in-Bm0}
\begin{split}
h^{\ep}(x, \a) &\leq h(x, \a) + |h^{\ep}(x, \a) - h(x, \a)|
\leq h(x_{m_{0}}, \a_{m_{0}}) + \frac{h(x_{m_{1}}, \a_{m_{1}}) - h(x_{m_{0}}, \a_{m_{0}})}{4} \\
&=  \frac{h(x_{m_{1}}, \a_{m_{1}}) + 3h(x_{m_{0}}, \a_{m_{0}})}{4} < \frac{3h(x_{m_{1}}, \a_{m_{1}}) + h(x_{m_{0}}, \a_{m_{0}})}{4}
\end{split}
\end{equation}
for every $x \in \overline{B}_{|x_{m_{0}}|}$, $\a \in \mathbb{S}^{n-1}$, and $\ep \leq \ep_{0}$.
Since $B_{P'} \subseteq B_{|x_{m_{0}}|}$, \eqref{eqn:high-at-m-1} and \eqref{eqn:low-in-Bm0} imply that $h^{\ep}(x, \a) \leq h^{\ep}(x_{m_{1}}, \a_{m_{1}})$ in $B_{P'}$.
Therefore, $h^{\ep}$ satisfies \eqref{eqn:max-outside-deg-eps} for every $\ep \leq \ep_{0}$.  

Recall that $\psi$ is constant outside $B_{P}$. Then, by \eqref{eqn:R3-R2-applied} and \eqref{eqn:max-outside-deg-eps}, we know that for every $\ep \leq \ep_{0}$, the function $e^{\psi(\nab \phi(x))} $ is locally constant around $x^{\ep}$. Therefore, $(x^{\ep}, \a^{\ep})$ is also a local maximum point for the incremental quotient $\phi(x+\ep\a)+\phi(x-\ep\a)-2\phi(x)$. 
Moreover, outside $B_{R}$ the function $V^{\j}-c_{q,\j}+q-c_{q}$ is convex as it coincides with $V^{\j}-c_{q,\j}-c_{q}$. So, proceeding as in the proof of Theorem~\ref{thm:Caff} (cf. Remark~\ref{rmk:local}), we conclude that \eqref{eqn:ts-rho} is also proved in the case that $h$ is not guaranteed to achieve a maximum in $\Om$.

In order to remove the smoothness assumptions on $V$ and $q$, we approximate $V$ and $q$ by convolution (adding a small constant to ensure these approximations define probability measures). Then, from what we have shown above, the approximate transports are all globally and uniformly Lipschitz. Thanks to the stability of optimal transports, passing to the limit, we prove \eqref{ts:lip-transp}.
\end{proof}

\section{Bounded Perturbations in $1$-Dimension and in the Radially Symmetric Case: Proofs of Theorems~\ref{1-D noncompact} and~\ref{thm:sym}}
\label{sec:Bounded Perturbations}

Our goal now is to produce optimal global Lipschitz estimates under strong symmetry but weak regularity assumptions on our log-concave measures. Notice that when our perturbation is zero, we recover that our optimal transport is the identity map (cf. Remark~\ref{rmk: Caff no ideal}). We begin in $1$-dimension and with a technical lemma relating the behavior of our convex base and the cumulative distribution function of the log-concave probability measure it defines.

\begin{lemma}
\label{lemma:log-V}
Let $V : \R \to \R$ be a convex function such that $e^{-V(x)} \,dx \in \Prob(\R)$ and $x_{0} \in \R$ be such that $V(x_{0}) = \inf_{\R} V$. Define $\Phi$, $\Psi : \R \to (0,1)$ by
\begin{equation}
\label{defn:Phi}
\Phi(x) := \int_{-\infty}^{x} e^{-V(t)} \,dt \quad\text{ and } \quad\Psi(x) := \int_{x}^{\infty} e^{-V(t)} \,dt=1-\Phi(x).
\end{equation}
Then,
\begin{equation}
\label{logPhi V ineq}
V(x)-V(y) \leq \log\Phi(y)-\log\Phi(x) \qquad \forall \,x \leq y \leq x_{0}
\end{equation}
and
\begin{equation}
\label{logPsi V ineq}
V(x)-V(y) \geq \log\Psi(y)-\log\Psi(x) \qquad \forall \,x_{0} \leq x \leq y.
\end{equation}
\end{lemma}

\begin{proof}
Since an analogous argument proves \eqref{logPsi V ineq}, we only show \eqref{logPhi V ineq}; in other words, we prove that the function $\log \Phi + V$ is nondecreasing in $(-\infty, x_0]$. 
Let $\hat{x} = \inf\{x:V(x)=V(x_0)\}$. The function $\log \Phi + V$ is clearly nondecreasing in $[\hat{x},x_0]$, whenever this interval is not a single point. Moreover, it is locally Lipschitz and its derivative is $e^{-V}/ \Phi + V'$. Hence, it suffices to show that the derivative is nonnegative in $(-\infty, \hat{x})$. Since $V'$ is nonincreasing in $(-\infty, \hat{x})$ and by the change of variables formula, we have that for a.e. $x\in (-\infty, \hat{x})$
$$V'(x)\Phi(x) \geq \int_{-\infty}^x V'(t) e^{-V(t)} \, dt =- e^{-V(x)},$$
which proves our claim.
\end{proof}

\begin{proof}[Proof of Theorem~\ref{1-D noncompact}]
By approximating $V$ with a sequence of convex functions $V_j \to V$ such that $e^{-V_{j}(x)} \, dx \in \mathcal{P}(\R)$ and that are finite on $\R$, we can assume that $V<\infty$ on $\R$. This reduction follows from the stability of optimal transport maps.  Recall that, as a consequence of the push-forward condition $T_\# \bigl(e^{-V(x)}\,dx\bigr)=e^{-V(x)-q(x)}\,dx$, $T$ satisfies the mass balance equation
\begin{equation}
\label{mass balance 1}
\int_{-\infty}^{x} e^{-V(t)} \,dt = \int_{-\infty}^{T(x)} e^{-V(t)-q(t)} \,dt,
\end{equation}
which can be also written as
\begin{equation}
\label{mass balance 2}
\int_{x}^{\infty} e^{-V(t)} \,dt = \int_{T(x)}^{\infty} e^{-V(t)-q(t)} \,dt
\end{equation}
since the measures $e^{-V(x)} \,dx$ and $e^{-V(x)-q(x)} \,dx$ have total mass $1$.
From \eqref{mass balance 1}, we deduce that $T$ is differentiable. 
Indeed, both the functions 
\[ F(x) := \int_{-\infty}^{x} e^{-V(t)} \,dt \qquad\text{and}\qquad G(x) := \int_{-\infty}^{x} e^{-V(t)-q(t)} \,dt \] 
are differentiable and their derivatives do not vanish. So, $T(x) = G^{-1} \circ F(x)$ is differentiable as well. Thus, differentiating with respect to $x$ and then taking the logarithm shows that
\begin{equation*}
\log(T'(x)) = -V(x)+V(T(x))+q(T(x)) \qquad \forall \,x\in \R.
\end{equation*}
Consequently,
\begin{equation}
\label{logMA with work}
V(T(x))-V(x)-\|q^{-}\|_{L^{\infty}(\R)} \leq \log(T'(x)) \leq V(T(x))-V(x)+\|q^{+}\|_{L^{\infty}(\R)}.
\end{equation}
On the other hand, \eqref{mass balance 1} implies that
\[ e^{-\|q^{+}\|_{L^{\infty}(\R)}} \int_{-\infty}^{T(x)} e^{-V(t)} \,dt \leq \int_{-\infty}^{x} e^{-V(t)} \,dt \leq e^{\|q^{-}\|_{L^{\infty}(\R)}} \int_{-\infty}^{T(x)} e^{-V(t)} \,dt \]
since $q \in L^{\infty}(\R)$.
Taking the logarithm and defining $\Phi$ as in \eqref{defn:Phi}, we see that
\begin{equation}
\label{eqn:logPhi-q ineq}
-\|q^{+}\|_{L^{\infty}(\R)} \leq \log \Phi(x) - \log \Phi(T(x)) \leq \|q^{-}\|_{L^{\infty}(\R)}.
\end{equation}
Analogously, from \eqref{mass balance 2}, we deduce that
\begin{equation}
\label{eqn:logPsi-q ineq}
-\|q^{+}\|_{L^{\infty}(\R)} \leq \log \Psi(x) - \log \Psi(T(x)) \leq \|q^{-}\|_{L^{\infty}(\R)}.
\end{equation}
We claim that
\begin{equation}
\label{eqn:Vx-VTx}
-\|q^{+}\|_{L^{\infty}(\R)} \leq V(T(x))-V(x) \leq \|q^{-}\|_{L^{\infty}(\R)} \qquad \forall \,x \in \R.
\end{equation}
To prove this claim, let $x_{0} \in \R$ be such that $V(x_{0}) = \inf_{\R} V$ and consider the sets
\[ E_{1} := \{ x : x \leq x_{0} \text{ and } T(x) \leq x_{0} \}
\qquad\text{and}\qquad
E_{2} := \{ x : x \geq x_{0} \text{ and } T(x) \geq x_{0} \}. \]
Applying \eqref{logPhi V ineq} in $E_{1}$ yields that 
\[ 0 \leq V(T(x))-V(x) \leq \log \Phi(x) - \log \Phi(T(x)) \] 
if $T(x) \leq x \leq x_{0}$ and   
\[ \log \Phi(x) - \log \Phi(T(x)) \leq V(T(x))-V(x) \leq 0 \]
whenever $x \leq T(x) \leq x_{0}$.
Therefore, \eqref{eqn:Vx-VTx} holds in $E_{1}$ by \eqref{eqn:logPhi-q ineq}. Similarly, applying \eqref{logPsi V ineq} gives us that \eqref{eqn:Vx-VTx} holds in $E_{2}$ by \eqref{eqn:logPsi-q ineq}. Now, we consider three cases: 

1. If $T(x_{0}) = x_{0}$, the monotonicity of $T$ implies that $E_{1} \cup E_{2} = \R$, and \eqref{eqn:Vx-VTx} holds in all of $\R$. 

2. If $T(x_{0}) > x_{0}$, then $E_{1} \cup E_{2} \cup E_{+} = \R$ where, thanks to the monotonicity of $T$, we have 
\[ E_{+} = \{x : x \leq x_{0} \text{ and } T(x) \geq x_{0} \} = [T^{-1}(x_{0}), x_{0}]. \]
Since $V$ attains its minimum at $x_{0}$, $V$ is decreasing on $(-\infty, x_{0}]$ and increasing on $[x_{0}, \infty)$. Consequently, 
\[ V(T^{-1}(x_{0}))-V(x_{0}) \leq V(x)-V(T(x)) \leq V(T(x_{0}))-V(x_{0}) \qquad \forall \,x \in E_{+}. \]
As $T^{-1}(x_{0}) \in E_{1}$ and $x_{0} \in E_{2}$, our above analysis shows that \eqref{eqn:Vx-VTx} holds in $E_{+}$.

3. If $T(x_{0}) < x_{0}$, an analogous argument to one used to prove case 2 demonstrates that $E_{1} \cup E_{2} \cup E_{-} = \R$ where $E_{-} = [x_{0}, T^{-1}(x_{0})]$ and proves \eqref{eqn:Vx-VTx} also in $E_{-}$.

Therefore, by \eqref{logMA with work} and \eqref{eqn:Vx-VTx}, we deduce \eqref{ts: T-lip-q}.
\end{proof}

\begin{remark}
\label{rmk:lin-est} 
From the numerical inequality $|\log(x)| \geq x-1$, which holds for $x\in [0,e^2]$, we see that if $\phi$ is the potential associated to $T$ in Theorem~\ref{1-D noncompact}, then provided that $\|q\|_{L^{\infty}(\R)} \leq 1$, there exists a constant $C > 0$ such that
\begin{equation*}
\label{eqn:ts-linearized}
\|\phi''-1\|_{L^{\infty}(\R)} \leq C\|q\|_{L^{\infty}(\R)}.
\end{equation*}
\end{remark}

We now move to the radially symmetric case in $n$-dimensions.

\begin{proof}[Proof of Theorem~\ref{thm:sym}]
Let $\overline V, \, \overline q: \R \to \R \cup \{ \infty\}$ be two functions such that
$\overline V=\overline q=\infty$ on $(-\infty,0)$, and $V(x) = \overline V (|x|)$ and $q(x) = \overline q(|x|)$ for every $x\in \R^n$.  Now, consider the function
\[ T(x) := \tilde{T}(|x|)\frac{x}{|x|} \]
where $\tilde{T}:\R\to\R$ is the optimal transport that takes $e^{-\overline V(r)}r^{n-1} \,dr$ to $e^{-\overline V(r)-\overline q(r)}r^{n-1} \,dr$. 

Set $\R^+:=[0,\infty)$.
We first claim that the optimal transport $\tilde T$ 
is Lipschitz and satisfies
\begin{equation}
\label{ts: T-lip-p-radial}
\|\log \tilde T'\|_{L^{\infty}(\R^{+})} \leq \|q^{+}\|_{L^{\infty}(\R^{+})}+\|q^{-}\|_{L^{\infty}(\R^{+})}.
\end{equation}
Indeed, let $\tilde{V} : \R \to \R \cup \{ \infty \}$ be defined by $\tilde{V}(r) = \overline V(r) - (n-1)\log r$ on $\R^{+}$ and infinity otherwise, and let $\tilde{q} =\overline q$ on $\R^{+}$ and zero elsewhere. Observe that $\tilde{V}$ is convex and $\tilde{q}$ is bounded. Hence, applying Proposition~\ref{1-D noncompact} with $V = \tilde{V}$ and $q = \tilde{q}$ proves \eqref{ts: T-lip-p-radial}.

We now conclude the proof.
Notice that $T$ is continuous. Furthermore, $T$ is an admissible change of variables from $e^{-V(x)} \,dx$ to $e^{-V(x)-q(x)} \,dx$. To see this, we show that for every bounded, Borel function $\varphi : \R^{n} \to \R$,
\begin{equation}
\label{eqn:cdv-radial}
\int_{\R^{n}} \varphi(T(x)) e^{-V(x)} \, dx = \int_{\R^{n}} \varphi(x) e^{-V(x)-q(x)} \, dx.
\end{equation}
The formula \eqref{eqn:cdv-radial} can be rewritten, using polar coordinates and the definition of $T$, as
\[ \int_{0}^{\infty} \int_{\mathbb{S}^{n-1}} \varphi(\tilde{T}(r)\a)\, d \H^{n-1}(\a)\, e^{-\overline V(r)}r^{n-1} \, dr =\int_{0}^{\infty} \int_{\mathbb{S}^{n-1}} \varphi(r \a)\, d \H^{n-1}(\a)\, e^{-\overline V(r)-\overline q(r)}r^{n-1} \, dr, \]
which is, in turn, satisfied if we use the test function $\overline \varphi(r) = \int_{\mathbb{S}^{n-1}} \varphi(r\a) \, d \H^{n-1}(\a)$ and recall the definition of $T$. 

Now, let $\xi\in \mathbb{S}^{n-1}$ and $x \in \R^{n}\setminus \{0\}$. Since $\tilde{T}(0)=0$, we observe that
\begin{equation*}
\begin{split}
\nab T(x)[\xi] &= \big{[} \xi|x|^{-1} - x|x|^{-3}\langle x, \xi \rangle\big{]}\tilde{T}(|x|) + x|x|^{-2}\tilde{T}'(|x|)\langle x, \xi \rangle \\
&= \big{[} \xi - x|x|^{-2}\langle x, \xi \rangle\big{]}\tilde{T}'(t) + x|x|^{-2}\tilde{T}'(|x|)\langle x, \xi \rangle
\end{split}
\end{equation*}
where $t \in (0, |x|)$. By \eqref{ts: T-lip-p-radial}, we deduce that
\[ e^{-\|q^{+}\|_{L^{\infty}(\R^{n})}-\|q^{-}\|_{L^{\infty}(\R^{n})}} \leq \langle \xi, \nab T(x)[\xi] \rangle \leq e^{\|q^{+}\|_{L^{\infty}(\R^{n})}+\|q^{-}\|_{L^{\infty}(\R^{n})}}, \]
which proves \eqref{ts: S-lip}. To conclude, we show that $T$ is the optimal transport taking $e^{-V(x)} \,dx$ to $e^{-V(x)-q(x)} \,dx$. Let $\tilde\phi : \R^{+} \to \R^{+}$ be the convex potential associated to $\tilde{T}$. By construction, $T(x) = \nab(\tilde\phi(|x|))$ and $\tilde\phi(|x|)$ is a convex function. Since optimal transports are characterized by being gradients of convex functions, $T$ is the optimal transport taking $e^{-V(x)} \,dx$ to $e^{-V(x)-q(x)} \,dx$.
\end{proof}

\section{Appendix}
\label{sec:Appendix}

We now show that the linear bound in Remark \ref{rmk:lin-est} is specific to the $1$-dimensional case.

\begin{proposition}
Let $n \in \N$ and $V(x) = |x|^{2}/2+(n/2)\log (2\pi)$, so that $e^{-V}$ is the standard Gaussian density in $\R^n$. Then, for every $C > 0$, there exists a bounded, continuous perturbation $p$ such that $\|p\|_{L^{\infty}(\R^{n})} \leq 1$ and $e^{-V(x)-p(x)} \, dx \in \Prob(\R^n)$ and the optimal transport $T = \nab \phi$ that takes $e^{-V(x)} \, dx$ to $e^{-V(x)-p(x)} \, dx$ satisfies 
\begin{equation*}
\|D^{2}\phi-\Id\|_{L^{\infty}(\R^{n})} > C\|p\|_{L^{\infty}(\R^{n})}.
\end{equation*}
\end{proposition}

\begin{proof}
Suppose, to the contrary, that for every bounded, continuous function $p:\R^n\to \R$ with $\|p\|_{L^{\infty}(\R^{n})} \leq 1$, the optimal transport $T = \nab \phi$ that takes $e^{-V(x)} \, dx$ to $e^{-V(x)-p(x)} \, dx$ satisfies 
\begin{equation}
\label{eqn:1-to-n}
\|D^{2}\phi-\Id\|_{L^{\infty}(\R^{n})} \leq C_0\|p\|_{L^{\infty}(\R^{n})}
\end{equation}
for some $C_0 > 0$.
In particular, let $q \in L^{\infty}(\R^{n}) \cap C^0(\R^n)$, and for all $\ep \geq 0$, define $c_{\ep}$ by
\begin{equation*}
\label{eps-cts} 
e^{c_{\ep}} = \int_{\R^{n}} e^{-V(x)-\ep q(x)} \, dx.
\end{equation*}  
By construction, $e^{-V(x)-\ep q(x)-c_{\ep}} \, dx \in \mathcal{P}(\R^n)$.
Thus, let $\phi_{\ep}$ be the potential associated to the optimal transport that takes $e^{-V(x)} \, dx$ to $e^{-V(x)-\ep q(x)-c_{\ep}} \, dx$, and remember that $\phi_{\ep}$ solves the Monge-Amp\`{e}re equation 
\begin{equation} 
\label{ep-MA}
\det D^{2}\phi_{\ep} = e^{-V+V(\nab \phi_{\ep})+\ep q(\nab \phi_{\ep})+c_{\ep}}.
\end{equation}
Note that $c_{\ep} \to 0$ as $\ep \to 0$. Also,
since
$$
|c_{\ep}'| = \bigg{|}\frac{(e^{c_{\ep}})'}{e^{c_{\ep}}}\bigg{|} = \bigg{|}\int_{\R^{n}} -q(x)e^{-V(x)-\ep q(x)-c_{\ep}} \, dx\bigg{|} \leq \|q\|_{L^{\infty}(\R^{n})},
$$
$c_{\ep}$ is Lipschitz as a function of $\ep$ and
\begin{equation}
\label{c_ep-Lip}
\frac{|c_\ep|}{\ep} \leq  \|q\|_{L^{\infty}(\R^{n})}.
\end{equation}
In addition, by the dominated convergence theorem, 
\begin{equation}
\label{c_ep-Lip 0}
c_{\ep}' \to \iota_{q} := \int_{\R^{n}} - q(x)e^{-V(x)} \, dx\qquad\text{as $\ep \to 0$.}
\end{equation}
Without loss of generality, we assume that $\phi_{\ep}(0)=0$. 
Now, define
\begin{equation*}
\label{psi-eps} 
\psi_{\ep}(x) := \frac{\phi_{\ep}(x)-|x|^{2}/2}{\ep}.
\end{equation*} 
By \eqref{eqn:1-to-n} applied to $p = \ep q+c_{\ep}$ and \eqref{c_ep-Lip}, we see that if $\ep \leq \frac{1}{2\|q\|_{L^{\infty}(\R^{n})}}$, then
\begin{equation}
\label{Hess psi-ep bd} 
\|D^{2}\psi_{\ep}\|_{L^{\infty}(\R^{n})} \leq (C_0+1)\|q\|_{L^{\infty}(\R^{n})}.
\end{equation}
Recall that, for any $n \times n$ matrix $A$, there exists a $K > 0$, depending only on $\|A\|$, such that for all $\ep$ sufficiently small $|\log \det(\Id+\ep A) - \ep\trace A| \leq  \ep^{2}K$. Therefore, there exist an $\ep_{0}>0$ and a collection of functions $g_{\ep}$ with
\begin{equation}
\label{logdet-trace}
\sup_{\ep \leq \ep_{0}} \|g_{\ep}\|_{L^{\infty}(\R^{n})} < \infty 
\end{equation}
such that for all $\ep \leq \ep_{0}$,
\[ \ep\Delta \psi_{\ep}(x)+\ep^{2}g_{\ep}(x) =\log\det (\Id + \ep D^{2}\psi_\ep) = \log\det D^{2}\phi_{\ep}. \]
Thus, by \eqref{ep-MA} and our choice of $V$,
\begin{equation}
\label{eqn:FTC}
\begin{split}
\Delta \psi_{\ep}(x) + \ep g_{\ep}(x) 
&= \frac{V(\nab \phi_{\ep}(x)) - V(x) + \ep q(\nab \phi_{\ep}(x)) + c_{\ep}}{\ep} 
\\
&= \int_{0}^{1} \langle (1-t)\nab \phi_{\ep}(x) + tx, \nab \psi_{\ep}(x) \rangle \, dt + q(\nab \phi_{\ep}(x)) + \frac{c_{\ep}}{\ep}
\\
&= \langle x, \nabla \psi_{\ep}(x) \rangle + \frac{\ep}{2}|\nabla \psi_{\ep}(x)|^2 + q(\nab \phi_{\ep}(x)) + \frac{c_{\ep}}{\ep}.
\end{split}
\end{equation}

We claim that, up to a subsequence, there exists a function $\psi_{0} \in C^{1,1}_{\rm{loc}}(\R^{n})$ such that $\psi_{\ep} \to \psi_{0}$ in $C^{1}_{\rm{loc}}(\R^{n})$ and $D^{2}\psi_{\ep} \rightharpoonup D^{2}\psi_{0}$ weakly-$\ast$ in $L^{\infty}(\R^{n})$ as $\ep \to 0$.
To this end, by Arzel\`{a}-Ascoli, it suffices to show that $\psi_{\ep}$ are locally bounded in $C^{1,1}$. Since $\psi_{\ep}(0) = 0$, by \eqref{Hess psi-ep bd}, it is enough to prove that
\begin{equation}
\label{grad psi-ep bd} 
\liminf_{\ep \to 0} |\nab \psi_{\ep}(0)| < \infty.
\end{equation}
Assume, to the contrary, that $\lim_{\ep \to 0} |\nab \psi_{\ep}(0)| = \infty$. Notice that \eqref{eqn:FTC} implies that for all $\ep \leq \ep_0$ and $x\in \R^n$,
\[\begin{split}
\bigg{|} \int_{0}^{1} \langle & (1-t)\nab \phi_{\ep}(x) + tx, \nab \psi_{\ep}(0) \rangle \, dt \bigg{|} 
\\
&\leq \bigg{|} \int_{0}^{1} \langle (1-t)\nab \phi_{\ep}(x) + tx, \nab \psi_{\ep}(x) \rangle \, dt \bigg{|} + \big(|\nab \phi_{\ep}(x)| + |x| \big)| \nab \psi_{\ep}(0)-  \nab \psi_{\ep}(x)| 
\\
&\leq |\Delta \psi_{\ep}(x)| + \ep |g_{\ep}(x)| + |q(\nab \phi_{\ep}(x))| + \frac{|c_{\ep}|}{\ep} + \big{(}|\nab \phi_{\ep}(x)| + |x| \big{)} |x| \sup_{\ep\leq\ep_{0}} \|D^{2} \psi_{\ep}\|_{L^{\infty}(\R^{n})}.  
\end{split}\]
Let $\a_{\ep} = \nab\psi_{\ep}(0)/|\nab \psi_{\ep}(0)| \in \mathbb{S}^{n-1}$, and note that up to subsequences $\a_{\ep} \to\a_{0} \in \mathbb{S}^{n-1}$ as $\ep \to 0$. Furthermore, let $\eta \in C^{\infty}_{c}(B_{1/2}(\a_{0}))$ be a nonnegative function that integrates to one. Then, by \eqref{c_ep-Lip}, we deduce that 
\begin{equation}
\label{eqn:pippo}
\begin{split}
|\nab \psi_{\ep}(0)| & \int_{\R^{n}} \bigg{|}\int_{0}^{1} \langle (1-t)\nab \phi_{\ep}(x) + tx, \a_{\ep} \rangle \, dt \bigg{|}\eta(x) \, dx \\
&= \int_{\R^{n}} \bigg{|} \int_{0}^{1} \langle (1-t)\nab \phi_{\ep}(x) + tx, \nab \psi_{\ep}(0) \rangle \, dt \bigg{|} \eta(x) \, dx\\
&\leq \sup_{\ep \leq \ep_{0}} \ep\|g_{\ep}\|_{L^{\infty}(\R^{n})} + 2\|q\|_{L^\infty(\R^{n})} \\
&\hspace{2.0cm}+ \sup_{\ep \leq \ep_{0}} \|D^{2} \psi_{\ep}\|_{L^\infty(\R^{n})}\int_{\R^{n}} \big{(}|\nab \phi_{\ep}(x)||x| + |x|^{2} + 1\big{)}\eta(x) \, dx.
\end{split}
\end{equation}
Recall that $D^{2}\phi_{\ep}$ converges uniformly to the identity matrix by \eqref{eqn:1-to-n} applied to $\phi_{\ep}$ and $\ep q$. By the stability and uniqueness of optimal transports, $\nab \phi_{\ep}$ converges locally uniformly to the identity map as $\ep \to 0$. In particular,
$ |\nab \phi_{\ep}(x)| \leq  2$ for every $x \in B_{1/2}(\a_{0})$ and $\ep$ sufficiently small, and we obtain that
\[ \lim_{\ep \to 0}  \int_{\R^{n}} \bigg{|}\int_{0}^{1} \langle (1-t)\nab \phi_{\ep}(x) + t x, \a_{\ep} \rangle \, dt \bigg{|} \eta(x) \, dx = \int_{\R^{n}} \langle x, \a_{0} \rangle \eta(x) \, dx \geq \frac{1}{2} \]
by dominated convergence.
Thus, taking the limit in \eqref{eqn:pippo} and noticing that the right-hand side is bounded as $\ep \to 0$ thanks to \eqref{Hess psi-ep bd} and \eqref{logdet-trace},
we see that
\[ \infty = \lim_{\ep \to 0} |\nab \psi_{\ep}(0)|\bigg{|}\int_{0}^{1} \langle (1-t)\nab \phi_{\ep}(x) + tx, \a_{\ep} \rangle \, dt \bigg{|} 
<\infty, \]
which, being impossible, proves \eqref{grad psi-ep bd} and shows that $\psi_{\ep} \to \psi_{0}$ in $C^{1}_{\rm{loc}}(\R^{n})$ and $D^{2}\psi_{\ep} \rightharpoonup D^{2}\psi_{0}$ weakly-$\ast$ in $L^{\infty}(\R^{n})$ as $\ep \to 0$ for some function $\psi_{0} \in C^{1,1}_{\rm{loc}}(\R^{n})$.

Now, reformulating \eqref{eqn:FTC}, we see that for any $\eta \in C_{c}^{\infty}(\R^{n})$,
\begin{equation}
\label{ep-weak eqn}
\int_{\R^{n}} \Bigl(\Delta \psi_{\ep}(x) + \ep g_{\ep}(x) - q(\nab \phi_{\ep}(x)) - \frac{c_{\ep}}{\ep}\Bigr)\eta(x) \, dx = \int_{\R^{n}} \Bigl(\langle x, \nab \psi_{\ep}(x) \rangle + \frac{\ep}{2}|\nab \psi_{\ep}(x)|^{2}\Bigr)\eta(x) \, dx.
\end{equation}
Thus, recalling \eqref{c_ep-Lip 0} and that $q$ is continuous, we can pass to the limit and obtain that
\begin{equation*}
\int_{\R^{n}} \bigl(\Delta \psi_{0}(x) - \langle x, \nab \psi_{0}(x) \rangle\bigr)\eta(x) \, dx = \int_{\R^{n}} \bigl(q(x)+\iota_{q}\bigr)\eta(x) \, dx
\end{equation*}
for all $\eta \in C^{\infty}_{c}(\R^{n})$.
Since $q$ was arbitrary, we have shown that for every {$q \in L^{\infty}(\R^{n}) \cap C^0(\R^n)$}, there exists a function $\psi_0 \in C^{1,1}_{\rm{loc}}(\R^{n})$ solution to
\begin{equation}
\label{elliptic eqn}
\Delta \psi_{0}(x) - \langle x , \nab \psi_{0}(x) \rangle = q(x) + \iota_{q}.
\end{equation}

We now show that this is impossible. Recall that there exists a bounded, continuous $g$ and $\psi \in C^{1,\a}_{\rm{loc}}(B_{2}) \cap C^{\infty}_{}(B_{2} \setminus \{0\})$, for any $\a \in (0,1)$, such that $\Delta \psi(x) = g(x)$ in $B_{2}$, yet $\psi \notin C^{1,1}(B_{2})$. In particular, $\lim_{x \to 0} |D^{2}\psi(x)| = \infty$. (See \cite[Chapter 3]{hl}.)
Define 
\[ h(x) := \begin{cases} g(x) - \langle x , \nab \psi(x) \rangle &x \in B_{1}\\
g(x/|x|) - \langle x/|x| , \nab \psi(x/|x|) \rangle &x \in \R^{n} \setminus B_{1},
\end{cases} \]
and observe that, since $\psi \in C^{1,\a}_{\rm{loc}}(B_{2})$ {and $g$ is bounded and continuous, $h \in L^{\infty}(\R^{n}) \cap C^0(\R^n)$}. By construction, there exists a $\psi_{0} \in C^{1,1}_{\rm{loc}}(\R^{n})$ that solves \eqref{elliptic eqn} with $q = h$. Then, for $\psi_{1} := \psi_{0} - \psi$ we have that $\Delta \psi_{1}(x) - \langle x , \nab \psi_{1}(x) \rangle = \iota_{h}$ in $B_{1}$. Thus, $\psi_{1} \in C^{\infty}(B_{1})$ by elliptic regularity, a contradiction since $\psi \notin C^{1,1}_{\rm{loc}}(B_{1})$ and $\psi_{0} \in C^{1,1}(B_{1})$.
\end{proof}



\begin{thebibliography}{99}

\bibitem{adm}
	\newblock S. Alesker, S. Dar, and V. Milman, 
	\newblock {\em A remarkable measure preserving diffeomorphism between two convex bodies in $\R^{n}$}, 
	\newblock Geom. Dedicata {\bf 74}(2) (1999), 201-212.

\bibitem{b}
	\newblock Y. Brenier, 
	\newblock {\em Polar factorization and monotone rearrangement of vector-valued functions},
	\newblock Comm. Pure Appl. Math. {\bf 44}(4) (1991), 365-417.

\bibitem{c1} L. A. Caffarelli,
	\newblock {\em The regularity of mappings with a convex potential},
	\newblock J. Amer. Math. Soc. {\bf 5}(1) (1992), 99-104.

\bibitem{c2} 
	\newblock L. A. Caffarelli,
	\newblock {\em Monotonicity properties of optimal transportation and the FKG and related inequalities}, 
	\newblock Comm. Math. Phys. {\bf 214}(3) (2000), 547-563.

\bibitem{c2.1} 
	\newblock L. A. Caffarelli,
	\newblock {\em Erratum: Monotonicity properties of optimal transportation and the FKG and related inequalities}, 
	\newblock Comm. Math. Phys. {\bf 225}(2) (2002), 449-450.

\bibitem{c-e}
	\newblock D. Cordero-Erausquin, 
	\newblock {\em Some applications of mass transport to Gaussian-type inequalities},
	\newblock Arch. Rat. Mech. Anal. {\bf 161}(3) (2002), 257-269.
	
\bibitem{cor}
	\newblock D. Cordero-Erausquin, M. Fradelizi, and B. Maurey,
	\newblock {\em The (B) conjecture for the Gaussian measure of dilates of symmetric convex sets and related problems},
	\newblock J. Funct. Anal. {\bf 214}(2) (2004), 410-427.
	

\bibitem{dpf} G. De Philippis and A. Figalli,
	\newblock {\em The Monge-Amp\`{e}re equation and its link to optimal transportation},
	\newblock Bull. Amer. Math. Soc. (N.S.) {\bf 51}(4) (2014), 527-580.
	
	
\bibitem{fbook} A. Figalli,
	\newblock {"The Monge-Amp\`ere Equation and its Applications"},
	\newblock Z\"urich Lectures in Advanced Mathematics, to appear.	
%
%
     
\bibitem{hl} 
     \newblock Q. Han and F. Lin, 
     \newblock {"Elliptic Partial Differential Equations"}, 
     \newblock 2nd Edition, Courant Lecture Notes in Mathematics, Vol. 1, New York University, Courant Institute of Mathematical Sciences, New York, American Mathematical Society, Providence, 1997.     

\bibitem{har}
	\newblock G. Harg\'e,
	\newblock {\em A convex/log-concave correlation inequality for Gaussian measure and an application to abstract Wiener spaces},
	\newblock Probab. Theory Related Fields {\bf 130}(3) (2004), 415-440.

\bibitem{KimMil}
	\newblock Y.-H. Kim and E. Milman,
	\newblock {\em A generalization of Caffarelli's contraction theorem via (reverse) heat flow},
	\newblock Math. Ann. {\bf 354}(3) (2012), 827-862.


\bibitem{kla}
	\newblock B. Klartag,
	\newblock {\em Marginals of geometric inequalities},
	\newblock In: "Geometric Aspects of Functional Analysis", V. D. Milman and G. Schechtman (eds.), Lecture Notes in Mathematics, Vol. 1910, Springer Berlin, 2007, 133-166.

\bibitem{p} 
	\newblock A. V. Pogorelov,
	\newblock {\em The regularity of generalized solutions of the equation $\det(\partial^{2}u/\partial x_{i}\partial x_{j}) = \vphi(x_{1}, x_{2},..., x_{n}) > 0$}, 
	\newblock (Russian) Dokl. Akad. Nauk SSSR {\bf 200} (1971), 534-537.

\bibitem{v} 
	\newblock C. Villani, 
    \newblock {"Optimal Transport, Old and New"}, 
    \newblock Grundlehren des mathematischen Wissenschaften [Fundamental Principles of Mathematical Sciences], Vol. 338, Springer-Verlag Berlin Heidelberg, 2009.

\end{thebibliography}
\end{document}